\documentclass[A4,twoside,12pt,leqno]{article}
\usepackage{paper}
\begin{document}
\maketitle\vspace{-6mm}

\begin{abstract}
{\it We show that all knots up to $6$ crossings can be represented by polynomial knots of degree at most $7$, among which except for $5_2, 5_2^*, 6_1, 6_1^*, 6_2, 6_2^*$ and $6_3$ all are in their minimal degree representation. We provide concrete polynomial representation of all these knots. Durfee and O'Shea had asked a question: Is there any $5$ crossing knot in degree $6$? In this paper we try to partially answer this question. For an integer $d\geq2$, we define a set $\pdt$ to be the set of all polynomial knots given by $t\mapsto\fght$ such that $\deg(f)=d-2,\,\deg(g)=d-1$\hskip1.5mm and\hskip1.5mm$\deg(h)=d$. This set can be identified with a subset of $\rtd$ and thus it is equipped with the natural topology which comes from the usual topology $\rtd$. In this paper we determine a lower bound on the number of  path components of $\pdt$ for $d\leq 7$. We define path equivalence between polynomial knots in the space $\pdt$ and show that path equivalence is stronger than the topological equivalence.}


\end{abstract}

%%%%%%%%%%%%%%%%%%%%%%%%%%%%%%%%%%%%%%%%%%%%%%%%%%%%%%%%%%%%%%%%%%%%%%%%%%%%%%%%%%%%%%
\section{Introduction}\label{sec1}

\noindent\hskip6mm The idea of representing a long knot by polynomial embeddings was discussed by Arnold \cite{va1}. Later as an attempt to settle a long lasting conjecture of Abhyankar \cite{ssa} in algebraic geometry Shastri \cite{ars} proved that every long knot is ambient isotopic to an embedding given by $t\mapsto\fght$ where $f, g$ and $h$ are real polynomials. These kind of embeddings are referred as polynomial knots.\vo

In his paper Shastri produced a choice of very simple polynomials $f, g$ and $h$ to represent the {\em trefoil knot} and the {\em figure eight knot}. He was hoping that once there are more examples available to represent various knot types, the conjecture of Abhyankar may be solved. This motivated the study of polynomial knots in a more rigorous and constructive manner. Explicit examples were constructed to represent a few classes of knots such as {\em torus knots} (see \cite{rm3} and \cite{rs2}) and {\em two bridge knots} (see \cite{pm1} and \cite{pm4}). To make the polynomials as simple as possible the notions of degree sequence and the minimal degree sequence were introduced. The minimality was with respect to the lexicographic order in $\mathbb{N}^3$. In this respect minimizing a degree sequence of a knot became a concern.\vo

Around the same time Vassiliev \cite{va2} studied and discussed the topology of the space $\vd$ for $d\in\mathbb{N}$, where $\vd$ is the space (with a natural topology coming from $\mathbb{R}^{3d-3}$) of all polynomial knots $t\mapsto\fght$ such that $f, g$ and $h$ are monic polynomials, each of degree $d$ without constant terms. Later, Durfee and O'Shea \cite{do} studied the space $\kd$ for $d\in\mathbb{N}$, where the space $\kd$ is the space of all polynomial knots $t\mapsto\fght$ such that the highest degree among the degrees of $f, g$ and $h$ is exactly equal to $d$. For a nontrivial polynomial knot in $\kd$, by composing it with a suitable orientation preserving linear transformation, we get a polynomial knot $\phi=\fgh$ such that all the component polynomials have degree $d$. On the other hand if $\phi=(f,g,h)$ be a polynomial knot with $f,g,h$ have same degree $d$, by composing $\phi$ with a linear transformation of the form $\xyz\mapsto (x-\alpha z, y-\beta z, z)$ ($\alpha$ and $\beta$ being some suitable real numbers), we get a polynomial knot $t\mapsto\big(f_1(t), g_1(t), h(t)\big)$ with $\deg(f_1)$ and $\deg(g_1)$ being at most $d-1$, which by further composing with a linear transformation of the type $(x,y,z)\mapsto (x-\gamma y, y, z)$ gives a polynomial knot $t\mapsto\big(f_2(t), g_1(t), h(t)\big)$ with $\deg(f_2)$ at most $d-2$. These transformations are orientation preserving and hence the new polynomial knot obtained upon the compositions is topologically equivalent to the old one. Thus, if a nontrivial polynomial knot belongs to the space $\kd$, for $d\geq1$, it is equivalent to a polynomial knot with degree sequence $(d_1,d_2,d_3)$ such that $d_1<d_2<d_3\leq d$.\vo 

For a particular knot type, determining a polynomial representation with least degree is still an unsolved problem.  Another important question that can be asked is: given any positive integer $d$ how many knots can be realized as a polynomial knot in degree $d$? It can be seen that for $d\leq 4$ there is only one knot, namely the unknot. There are three nonequivalent knots that can be realized for $d=5$, namely the unknot, the right hand trefoil and the left hand trefoil. Note that if a knot is realized in degree $d$, then it can be realized in degrees higher than $d.$ For degree $6$ we found an additional knot, namely the figure eight knot. In this connection Durfee and O'Shea asked: are there any $5$ crossing knot in degree $6$?  We note that there are only two knots with $5$ crossings denoted as $5_1$ and $5_2$ in the Rolfsen's table. Using a knot invariant known as {\em superbridge index}, we can prove that $5_1$ can not be represented in degree $6$. For $5_2$ knot the superbridge index is not known. We show that there exists a projection of $5_2$ knot given by $t\mapsto\fgt$ with $\deg(f)=4$ and $\deg(g)=5$
 but there is no generic choice of coefficients of the polynomials $f$ and $g$ such that $5_2$ knot has a polynomial representation in degree $6$. We conjecture that there are no $5$ crossing knots in degree $6$. This will be more clear once the superbridge index of all the possible $3$-superbridge knots in the list of Jin and Jeon \cite{jj1} is completely known. We show that all $5$ crossing knots and all $6$ crossing knots (including the composite knots) are realized in degree $7$. We look at the spaces $\pd$ and $\pdt,$ where $\pd$ is the space of all polynomial knots $t\mapsto\fght$ with $\deg(f)<\deg(g)<\deg(h)\leq d$ and $\pdt$ is the space of of all polynomial knots $t\mapsto\fght$ such that $\deg(f)=d-2,\,\deg(g)=d-1$\hskip1.5mm and\hskip1.5mm$\deg(h)=d$ and define two polynomial knots in these spaces to be path equivalent if they belong to the same path component in that space.\vo

This paper is organized as follows: In section \ref{sec2}, we discuss polynomial knots and introduce the spaces $\pd$ and $\pdt.$   We prove a few relevant results in connection with the polynomial representation of knots with a given crossing information. At the end of section 2, we show that for a generic choice of a regular projection $t\mapsto\fgt$ of the knot $5_2$ with $\deg(f)=4$ and $\deg(g)=5$ there does not exist a polynomial $h$ of degree $6$ such that $t\mapsto\fght$ is its polynomial representation, which partially answers the question asked by Durfee and O'Shea.\vo 

We divide the section 3 into five subsections. In section \ref{sec3.1}, we discuss the topology of the spaces $\pd$ and  $\pdt$ for $d\geq 2$. In section \ref{sec3.2}, we estimate the path components of  $\pd$ and $\pdt$ for $d\leq 4$. Sections \ref{sec3.3}, \ref{sec3.4} and \ref{sec3.5} are devoted towards estimating the path components of $\pdt$ for $d=5,6$ and $7$ respectively. We also provide polynomial knots belonging to each path components for each knot type and at the end of each subsections we summarize the number of path components in the form of a table. We conclude the paper in section \ref{sec4} by mentioning a few remarks for the spaces $\pdt$, for $d>7$ and discussing about how the different topologies on the set $\spk$ of all polynomial knots can be given using different stratification and how they affect the path components of the resulting space.

%%%%%%%%%%%%%%%%%%%%%%%%%%%%%%%%%%%%%%%%%%%%%%%%%%%%%%%%%%%%%%%%%%%%%%%%%%%%%%%%%%%%%%%%
\section{Polynomial knots}\label{sec2}

\begin{definition}
A {\sf long knot} is a proper smooth embedding $\phi:\ro\to\rt$ such that the map $t\mapsto\left\|\phi(t)\right\|$ is strictly monotone outside some closed interval of the real line and $\left\|\phi(t)\right\|\to\infty$ as $\lvert\hp t\hp\rvert\to\infty$.
\end{definition}

Using the stereographic projection $\pi:\st\setminus\{(0,0,0,1)\}\to\rt$, we can identify the one point compactification of $\rt$ with $\st$. Thus, by this identification, any long knot $\phi:\ro\to\rt$ has a unique extension as a continuous embedding $\tilde{\phi}:\so\rightarrow \st$ which takes the north pole of $\so$ to the north pole of $\st$. The map $\tilde{\phi}$ is a tame knot and it is smooth everywhere except possibly at the north pole where it may have an algebraic singularity (see \cite{do}, proposition 1). 

\begin{definition}\label{def2}
Two long knots $\phi,\hf\psi:\ro\to\rt$ are said to be {\sf topologically equivalent (simply, equivalent)} if there exist orientation preserving diffeomorphisms $F:\ro\to\ro$ and $H:\rt\to\rt$ such that $\psi=H\circ\phi\circ F$. 
\end{definition}

\begin{definition}
A {\sf diffeotopy} (respectively, {\sf homeotopy}) of $\rt$ is a continuous map $H:\zo\times\rt\to\rt$ such that: (i) $H_s=H(s,\hf\cdot\hf)$ is a diffeomorphism (respectively, homeomorphism) of\, $\rt$ for each $s\in\zo$ and (ii) $H_0$ is the identity map of $\rt$. 
\end{definition}

\begin{definition}\label{def4}
Two long knots $\tau,\sigma:\ro\to\rt$ are said to be {\sf ambient isotopic} if there exists a diffeotopy $H:\zo\times\rt\to\rt$ such that $\sigma=H_1\circ\tau$.
\end{definition}

For classical knots as tame embeddings of $\so$ in $\st$, the terms as in the definitions \ref{def2} and \ref{def4} can be defined using orientation preserving self-homeomorphisms of $\so$ and $\st$ and homeotopies of the ambient space $\st$. Using the standard results in topology, the following proposition is easy to prove.
 
\begin{proposition}\label{thm1}
For long knots $\phi, \psi:\ro\to\rt$, the following statements are equivalent:\vo
\noindent\hskip3.5mm i) The knots $\phi$ and $\psi$ are equivalent.\vo
\noindent\hskip2.25mm ii) The knots $\phi$ and $\psi$ are ambient isotopic.\vo
\noindent\hskip1mm iii) The extensions $\tilde{\phi}:\so\rightarrow\st$ and $\tilde{\psi}:\so\rightarrow\st$ are equivalent.\vo   
\noindent\hskip1.6mm iv) The extensions $\tilde{\phi}:\so\rightarrow\st$ and $\tilde{\psi}:\so\rightarrow\st$ are ambient isotopic.
\end{proposition}

\begin{definition}
A {\sf polynomial map} is a map $\phi:\ro\to\rt$ whose component functions are univariate real polynomials.
\end{definition}

\begin{definition}
A {\sf polynomial knot} is a polynomial map which is an embedding.
\end{definition}

A polynomial knot is a long knot. It has been proved (see \cite{ars}) that each long knot is topologically equivalent to some polynomial knot. Thus, each tame knot $\kappa:\so\to\st$ is ambient isotopic to the extension $\tilde{\phi}:\so\to \st$ of some polynomial knot $\phi:\ro\to\rt$.

\begin{definition}
A polynomial map $\phi=\fgh$ is said to have a {\sf degree sequence $(\hs d_1, d_2, d_3\hs)$} if\, $\deg(f)=d_1,\, \deg(g)=d_2$ and\, $\deg(h)=d_3$.
\end{definition}

\begin{definition}
A {\sf polynomial degree} of a polynomial map $\varphi:\ro\to\rt$ is the maximum of the degrees of its component polynomials.
\end{definition}

By composing with an orientation preserving tame\footnote{A tame polynomial automorphism is a composition of orientation preserving affine transformations and maps which add a multiple of a positive power of one row to an another row.} polynomial automorphism of $\rt$, a nontrivial polynomial knot $\phi$ with degree $d$ acquires the form $\sigma=\fgh$ such that $\fghio$ and none of the degree lie in the semigroup generated by the other two (see \cite{do}, section 5). For a sufficiently small $\varepsilon>0$, by adding $\varepsilon\hp t^{d-2},\,\varepsilon\hp t^{d-1}$\hskip1.5mm and\hskip1.5mm$\varepsilon\hp t^d$ in the respective components, one can make the degrees of $f,\,g$\hskip1.5mm and\hskip1.5mm$h$ respectively to be $d-2,\,d-1$\hskip1.5mm and\hskip1.5mm$d$ without changing the topological type of the knot. In other words, each polynomial knot $\phi$ of degree $d$ is topologically equivalent to a polynomial knot $\varsigma=\fghp$ with $\deg(f')=d-2,\,\deg(g')=d-1$\hskip1.5mm and\hskip1.5mm$\deg(h')=d$.

For any but fixed positive integer $d\geq 2$, consider a set $\ad$ of all polynomial maps $\fgh:\ro\to\rt$ with $\deg(f)\leq d-2$, $\deg(g)\leq d-1$ and $\deg(h)\leq d$. A typical element of this set would be a map $t\mapsto\abct$, where $a_i$'s, $b_i$'s and $c_i$'s are real numbers. The set $\ad$ can be identified with $\rtd$ and so it has a natural topology which comes from the usual topology of $\rtd$. Let $\pd$ be the set of all polynomial knots $\sigma=\fgh$ with $\fghio$ and let $\pdt$ be the set of all polynomial knots $\varsigma=\fghp$ with $\deg(f')=d-2,\,\deg(g')=d-1$\hskip1.5mm and\hskip1.5mm$\deg(h')=d$. Both $\pd$ and $\pdt$ are proper subsets of $\ad$; therefore, they have subspace topologies which comes from the topology of $\ad$. In other words, the spaces $\pd$ and $\pdt$ can be thought of as topological subspaces of $\rtd$ through the natural identification. Also, we may think the elements of the spaces $\pd$ and $\pdt$ as ordered $3d$-tuples of real numbers.

\begin{remark}
Note that\, $\spm\subsetneq\pn$,\, $\pnt\subsetneq\pn$\, and\, $\pmt\nsubseteq\pnt$\, for all $n>m\geq2$.
\end{remark}

\begin{remark}\label{rem1}
For any $n>m$ and any $\phi=\fgh$ in $\pmt$, there exists a sufficiently small $\varepsilon>0$ such that a polynomial knot $\phi_\varepsilon\in\pnt$ given by $t\mapsto\big(\hp\varepsilon\hf t^{n-2}+f(t),\hs\varepsilon\hf t^{n-1}+g(t),\hs\varepsilon\hf t^n+h(t)\hp\big)$ is topologically equivalent to the polynomial knot $\phi$.
\end{remark}

\begin{definition}
An ambient isotopy class $[\kappa]$ of a tame knot $\kappa:\so\to\st$ is said to have a {\sf polynomial representation in degree $d$} if there is a polynomial knot $\phi:\ro\to\rt$ of degree $d$ such that its extension $\tilde{\phi}:\so\to\st$ is ambient isotopic to $\kappa$. In this case, the polynomial knot $\phi$ is called a {\sf polynomial representation} of the knot-type $[\kappa]$.
\end{definition}

\begin{definition} 
A knot-type $[\kappa]$ is said to have {\sf polynomial degree $d$} if it is the least positive integer such that there exists a polynomial knot $\phi$ of degree $d$ representing the knot-type $[\kappa]$. In this case, the polynomial knot $\phi$ is called a {\sf minimal polynomial representation} of the knot-type $[\kappa]$.
\end{definition}

\begin{remark}\label{rem4}
Obviously, the polynomial degree of a knot-type is a knot invariant.
\end{remark}

\begin{remark}\label{rem5}
If a knot-type $[\kappa]$ is represented by a polynomial knot $\phi=\fgh$, then the knots $\mifgh,\, \fmgh,\, \fgmh$\, and\, $\minusfgh$\, represent the knot-type $[\kappa^*]$ of the mirror image of $\kappa$. Thus, a knot and its mirror image have same polynomial degree. This says that the polynomial degree can not detect the chirality of knots.
\end{remark}

Certain numerical knot invariants can be inferred from the polynomial degree of a knot and vice-versa. In this connection some known useful results are summarized in the following proposition.

\begin{proposition}\label{thm2}
For a classical tame knot $\kappa:\so\to\st$, we have the following:\vo
\noindent\hskip1mm 1)\hskip2.5mm $c[\kappa]\leq \frac{(d-2)(d-3)}{2}$\hp,\vo
\noindent\hskip1mm 2)\hskip2.5mm $b[\kappa]\leq \frac{(d-1)}{2}$\hw and\vo
\noindent\hskip1mm 3)\hskip2.5mm $s[\kappa]\leq \frac{(d+1)}{2}$\hp,\vo
\noindent where $c[\kappa],\, b[\kappa],\,s[\kappa]$ and $d$ denote respectively the crossing number, the bridge index, the superbridge index and the polynomial degree of the knot-type $[\kappa]$.
\end{proposition}

The first part of proposition \ref{thm2} can be proved using the Bezout's theorem. The proofs of the second and third parts are trivial. To get an idea about the proofs, one can refer the propositions $12,\,13$ and $14$ in \cite{do}.\vo 

From the first result as mentioned in proposition \ref{thm2}, it is clear that in order to represent a knot with certain number of crossings the degree of its polynomial representation has some lower bound. The knots with same number of crossings may have different crossing patterns, i.e. over and under crossing information. The result below tells us how the degree relates to the nature of the crossings.

\begin{theorem}\label{thm3}
Let $t\mapsto\fgt$ be a regular projection of a long knot $\kappa:\ro\to\rt$, where $f$ and $g$ are real polynomials; $\deg(f)=n$ and $\deg(g)=n+1$. Suppose the crossing data $\kappa$ is such that there are $r$ changes from over/under crossings to under/over crossings as we move along the knot. Then there exists a polynomial $h$ with degree $d\leq\min\{\hs n+2,r\hs\}$ such that the polynomial map $t\mapsto\fght$ is an embedding which is ambient isotopic to $\kappa$.
 \end{theorem}

\begin{proof}
Let polynomials $f$ and $g$ be given by $f(t)=a_0 + a_1 t+\cdots + a_n t^n$ and $g(t)={b_0 + b_1 t+\cdots + b_{n+1} t^{n+1}}$. The double points of the curve $t\mapsto\fgt$ can be obtained by finding the real roots of the resultant $\varGamma$ of the polynomials 
\begin{eqnarray*}
F(s, t)&=& a_1 +a_2\hp(s+t)+\cdots + a_n\hp(s^{n-1}+ s^{n-2}t+\cdots+t^{n-1})\hw\mbox{and}\\ 
G(s, t)&=& b_1 +b_2\hp(s+t)+\cdots + b_{n+1}\hp(s^n+ s^{n-1}t+\cdots+t^n)\hp.
\end{eqnarray*} 
Let $[\hf a, b\hf]$ be an interval that contains all the roots of $\varGamma$. Let us call these roots as the crossing points. We divide the interval $[\hf a, b\hf]$ into sub intervals $a=a_0<a_1<\cdots<a_r=b$ in such a way that the $a_i$'s are not from the crossing points and within any sub interval $[\hf a_{i-1}, a_i\hf]$ all the crossing points are either under crossing points or over crossing points. Let $h_1(t)=\Pi_{i=1}^r (t-a_i)$. Clearly, the map $h_1$ is a polynomial of degree $r$ that has opposite signs at under crossing and over crossing points and thus $\phi_1=(\hs f,\hs  g,\hs  h_1\hs)$ represents the knot $\kappa$.
 
For $i\in\{1,2,\ldots,n\}$, let $(s_i, t_i)$ with $s_i<t_i$ and $s_1<s_2<\cdots<s_n$ be pairs of parametric values at which the projection has double points. Let $C_{n+2}\hp t^{n+2}+C_{n+1}\hp t^{n+1}+\cdots+C_1\hp t$ be a polynomial of degree $n+2$, where $C_i$'s are unknowns which we have to find by solving the system
\begin{eqnarray*}
C_{n+2}\hp( t_1^{n+2}-s_1^{n+2})+C_{n+1}\hp( t_1^{n+1}-s_1^{n+1} )+\cdots+C_1\hp( t_1-s_1 )&=& e_1\\
C_{n+2}\hp( t_2^{n+2}-s_2^{n+2})+C_{n+1}\hp( t_2^{n+1}-s_2^{n+1} )+\cdots+C_1\hp( t_2-s_2 )&=& e_2\\
\vdots\hskip36.3mm\vdots\hskip38mm\vdots\hskip7.3mm & &\hskip1mm\vdots\\
C_{n+2}\hp( t_n^{n+2}-s_n^{n+2})+C_{n+1}\hp( t_n^{n+1}-s_n^{n+1} )+\cdots+C_1\hp( t_n-s_n )&=& e_n
\end{eqnarray*}
\noindent of $n$ linear equations in $n+2$ unknowns. The numbers $e_i$'s are arbitrary but fixed non-zero real numbers and they are positive or negative according to which the crossing is under crossing or over crossing. The conditions $f(s_i)=f(t_i)$ and $g(s_i)=g(t_i)$ for $i\in\{1,2,\ldots,n\}$ imply that the coefficient matrix\vo
\[
M=
\left[ 
{\begin{array}{cccc}
t_1^{n+2}-s_1^{n+2} & t_1^{n+1}-s_1^{n+1} & \cdots & t_1-s_1\\
t_2^{n+2}-s_2^{n+2} & t_2^{n+1}-s_2^{n+1} & \cdots & t_2-s_2\\
\vdots              & \vdots              &        & \vdots\\
t_n^{n+2}-s_n^{n+2} & t_n^{n+1}-s_n^{n+1} & \cdots & t_n-s_n\\
\end{array} } 
\right]
\]\vw
\noindent of the above system has rank $n$. This says that the system of linear equations has infinitely many solutions. Let $C_1=c_1, C_2=c_2,\ldots, C_{n+2}=c_{n+2}$ be any one of the solution. With this solution, the polynomial $h_2(t)=c_1\hp t+c_2\hp t^2+\cdots+c_{n+2}\hp t^{n+2}$ is such that the embedding $\phi_2=(\hs f,\hs  g,\hs  h_2\hs)$ represents the knot $\kappa$.
\end{proof}

In connection with polynomial representation of knots, the following questions are of interest namely:
\begin{question}
Given a knot $\kappa$, what is the least degree $d$ such that it has a polynomial representation in the space $\pd$?
\end{question}

\begin{question}
Given a positive integer $d$, what are the knots those can be represented in the space $\pd$?
\end{question}

\begin{question}
Given a positive integer $d$, estimate the number of path components of the spaces $\pd$ and $\pdt$.
\end{question}

We have tried to answer these questions for $d\leq7$. In general, all these problems become difficult and answering one helps in answering the other two.

\begin{proposition}\label{thm4}
The unknot is the only knot that can be represented as a polynomial knot in the space $\pd$ for $d\leq4$.
\end{proposition}

\begin{proof}
Let $\kappa$ be a knot which is represented by a polynomial knot in $\pd$ for $d\leq 4$. By proposition \ref{thm2}, the crossing number $c[\kappa]$ of the knot-type $[\kappa]$ satisfies the inequality $c[\kappa]\leq \frac{(d-2)(d-3)}{2}\leq \frac{(4-2)(4-3)}{2}=1$. Hence $\kappa$ must be the unknot.
\end{proof}

\begin{proposition}\label{thm5}
The unknot, the left hand trefoil and the right hand trefoil are the only knots those can be represented as polynomial knots in the space $\pf$.
\end{proposition}

\begin{proof}
Let $\kappa$ be a knot which is represented by a polynomial knot in $\pf$. By proposition \ref{thm2}, the crossing number $c[\kappa]$ of the knot-type $[\kappa]$ satisfies the inequality $c[\kappa]\leq \frac{(5-2)(5-3)}{2}=3$. Thus, $\kappa$ is either the unknot or one of the trefoil knot.
\end{proof}

The figure eight knot has a polynomial representation (figure 3, section \ref{sec3.4}) in the space $\pst$. However, since the knot $5_1$ is $4$-superbridge, so by proposition \ref{thm2}, it can not be represented in degree $6$. Regarding the knot $5_2$, we have the following:

\begin{proposition}\label{thm6}
There exist polynomials $f$ and $g$ of degrees $4$ and $5$ respectively such that the map $t\mapsto\fgt$ represents a regular projection of $5_2$ knot.
\end{proposition}

\begin{proof}
Consider a plane curve $C$ given by the parametric equation $t\mapsto(\hp t^4,\hp t^5\hp)$. This curve has an isolated singularity at the origin. For such plane curves, there are two important numbers that remain invariant under any formal isomorphisms of plane curves. The first one is the Milnor number $\mu$ and the other is the $\delta$ invariant (see \cite{na2}). For a single component plane curve they satisfy the relation $2\delta =\mu$.\vo
 
In the present case it turns out that the $\delta$ invariant is equal to $\frac{(5-1)(4-1)}{2}=6.$ The $\delta$ invariant of a plane curve which is singular at the origin measures the number of double points that can be created in a neighborhood of the origin. Note that, we have $\delta\geq 5$. Using a result of Daniel Pecker \cite{dp} we can deform the curve $C$ into a new curve $\tilde{C}$ given by $t\mapsto\abtf$ such that $\tilde{C}$ has $5$ real nodes and $1$ imaginary node. By continuity argument, we can choose the coefficients $a_i$'s and $b_i$'s such that the nodes occur in the order they are in the regular projection of the given knot.\vo 

In fact, we have found the curve $t\mapsto\big(\hp2(t - 2)(t + 4)(t^2 - 11),\hs t(t^2 - 6)(t^2 - 16)\hp\big)$ which represents a regular projection of $5_2$ knot as shown in the figure bellow:
\begin{figure}[H]
\begin{center}
\includegraphics[scale=0.32]{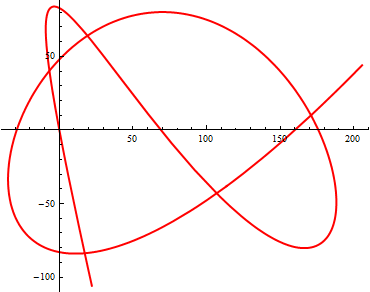}
\caption{Projection of $5_2$ knot}
\end{center}
\end{figure}\vspace{-13.5mm}
\end{proof}\vskip1.5mm

Let us consider the following sets:
\begin{align*}
U_1&=\big\{\hs\abf\in\ro^{11}\mid\mbox{a map}\hskip1.3mm t\mapsto\big(\hp a_4t^4+ a_3 t^3\\ &\hskip8.2mm+a_2 t^2+a_1t+a_0,\hp b_5t^5+b_4t^4+b_3t^3+b_2t^2+b_1t+b_0\hp\big)\hskip1.5mm\mbox{is a regular}\\ &\hskip9.2mm\mbox{projection of}\hskip1.3mm 5_2\hskip1.3mm \mbox{knot with}\hskip1.3mm 5\hskip1.3mm \mbox{double points}\hf\big\}.\\ 
U_2&=\big\{\hs\abf\in\ro^{11}\mid\mbox{a map}\hskip1.3mm t\mapsto\big(\hp a_4t^4+ a_3 t^3\\ &\hskip8.2mm+a_2 t^2+a_1t+a_0,\hp b_5t^5+b_4t^4+b_3t^3+b_2t^2+b_1t+b_0\hp\big)\hskip1.5mm\mbox{is a regular}\\ &\hskip9.2mm\mbox{projection of}\hskip1.3mm 5_2\hskip1.3mm \mbox{knot with}\hskip1.3mm 6\hskip1.3mm \mbox{double points}\hf\big\}.
\end{align*}
By proposition \ref{thm6}, it is easy to see that the set $U=U_1\cup U_2$ is nonempty. For any $\abf\in U$, we must have both $a_4$ and $b_5$ non zero, because otherwise by an application of the Bezout's theorem there would be less than five crossings (see \cite{do}, lemma 4) for the curve $t\mapsto\abtf$. For a projection $t\mapsto\fgt$ such that $f$ and $g$ are polynomials of degrees $4$ and $5$ respectively, we would like to find a polynomial $h$ of least possible degree such that $t\mapsto\fght$ represents $5_2$ knot. In this direction, we have the following theorem:

\begin{theorem}\label{thm7}
For a generic choice of a regular projection $t\mapsto\fgt$ with $\deg(f)=4$ and $\deg(g)=5$, there does not exist any polynomial $h$ of degree $6$ such that a polynomial map $t\mapsto\fght$ represents the knot $5_2$.
\end{theorem}

\begin{proof}
For $f(t)=\atfr$ and $g(t)=\btf$, let $t\mapsto\fgt$ be a regular projection of $5_2$ knot. Note that $a_4b_5\neq0$. Suppose $h(t)=\cts$ be a degree $6$ polynomial such that $t\mapsto\fght$ represents the knot $5_2$. By composing this embedding with a suitable affine transformation, we can assume that the coefficients $c_0$, $c_4$ and $c_5$ are zero. Thus, we can take $h(t)=c_6t^6+c_3t^3+c_2t^2+c_1t$. Note that the projection has either $5$ or $6$ double points. So, we consider the following cases:\vo

Case $i)$ If the projection has $5$ crossings:\\  
For $i\in\{1,2,\ldots,5\}$, let $(s_i, t_i)$ with $s_i<t_i$ and $s_1<s_2<\cdots<s_5$ be the pairs of parametric values at which the crossings occur in the curve $t\mapsto\fgt$. Since we want alternatively over and under crossings, so we should have $h(t_i)-h(s_i)$ is positive if $i$ is odd (i.e. the crossing is under crossing) and it is negative if $i$ is even (i.e. the crossing is over crossing). In other words, we have to find the coefficients $c_1,$ $c_2$, $c_3$ and $c_6$  such that for $i\in\{1,2,\ldots,5\}$ and $r_i\in\ro^+$, $h(t_i)-h(s_i)= r_i$ if $i$ is odd and $h(t_i)-h(s_i)= -r_i$ if $i$ is even. This gives us a system of $5$ linear equations in $4$ unknowns as follows:
\begin{eqnarray}
c_6(t_1^6-s_1^6)+c_3(t_1^3-s_1^3)+c_2(t_1^2-s_1^2)+c_1(t_1-s_1)&=& r_1\label{eq2.1}\\
c_6(t_2^6-s_2^6)+c_3(t_2^3-s_2^3)+c_2(t_2^2-s_2^2)+c_1(t_2-s_2)&=& -r_2\\
c_6(t_3^6-s_3^6)+c_3(t_3^3-s_3^3)+c_2(t_3^2-s_3^2)+c_1(t_3-s_3)&=& r_3\\
c_6(t_4^6-s_4^6)+c_3(t_4^3-s_4^3)+c_2(t_4^2-s_4^2)+c_1(t_4-s_4)&=& -r_4\\
c_6(t_5^6-s_5^6)+c_3(t_5^3-s_5^3)+c_2(t_5^2-s_5^2)+c_1(t_5-s_5)&=& r_5\label{eq2.5}
\end{eqnarray}
\noindent The rank of the coefficient matrix 
\[
A=
\left[ {\begin{array}{cccc}
t_1^6-s_1^6 & t_1^3-s_1^3 & t_1^2-s_1^2 & t_1-s_1 \\
t_2^6-s_2^6 & t_2^3-s_2^3 & t_2^2-s_2^2 & t_2-s_2 \\
t_3^6-s_3^6 & t_3^3-s_3^3 & t_3^2-s_3^2 & t_3-s_3 \\
t_4^6-s_4^6 & t_4^3-s_4^3 & t_4^2-s_4^2 & t_4-s_4 \\
t_5^6-s_5^6 & t_5^3-s_5^3 & t_5^2-s_5^2 & t_5-s_5 \\
\end{array} } \right]
 \]\vw
\noindent of the system of linear equations \ref{eq2.1} - \ref{eq2.5} is at most $4$. The system has a solution if and only if the rank of $A$ is equal to the rank of the augmented matrix $\tilde{A}$. In other words, the system of linear equations \ref{eq2.1} - \ref{eq2.5} has no solution if  $\det\hs(\tilde{A})\neq 0$. For $j\in\{1,2,\ldots,5\}$, let $A_j$ be a submatrix of $A$ obtained by deleting the $j^{th}$ row of $A$. It is easy to see that $$\det\hs(\tilde{A})=r_1\hs\det(A_{1})+r_2\hs\det(A_{2})+r_3\hs\det(A_{3}) +r_4\hs\det(A_{4})+r_5\det\hs(A_{5})\hp.$$ Note that for each $j$, $\det\hs(A_j)$ is an algebraic function of $t_i$'s and $s_i$'s which are actually analytic functions of the coefficients $a_k$'s of $f$ and the coefficients $b_k$'s of $g$. Thus $\det\hs(\tilde{A})$ is a non-constant analytic function of $a_k$'s, $b_k$'s and $r_k$'s. Hence the set $$V_1=\{\hf(\hp a_0,\ldots,a_4,b_0,\ldots,b_5, r_1,\ldots,r_5\hp)\in U_1\times(\mathbb{R}^+)^5\mid\det(\tilde{A})\neq0\hf\}$$ is an open and a dense subset of $U_1\times(\mathbb{R}^+)^5$. For any choice of an element in $V_1$, the system of linear equations \ref{eq2.1} - \ref{eq2.5} has no solution. Therefore, for a generic choice of a regular projection $t\mapsto\fgt$ with $5$ crossings and having $\deg(f)=4$ and $\deg(g)=5$, there does not exist any polynomial $h$ of degree $6$ such that $t\mapsto\fght$ represents the knot $5_2$.\vo

Case $ii)$ If the projection has $6$ crossings:\\  
For $i\in\{1,2,\ldots,6\}$, let $(s_i, t_i)$ with $s_i<t_i$ and $s_1<s_2<\cdots<s_6$ be the pairs of parametric values at which the crossings occur in the curve $t\mapsto\fgt$. Let $e=(e_1,e_2,\ldots,e_6)$ be a pattern such that this together with the projection $t\to\fgt$ describe the knot $5_2$, where $e_i$ is either $1$ or $-1$ according to which the $i^{th}$ crossing is under crossing or over crossing. Let $U_e$ be a set of elements $\abf$ of $U_2$ such that the projection $t\to\abtf$ together with the pattern $e$ describe the knot $5_2$. We want to find the values of the coefficients $c_1,$ $c_2$, $c_3$ and $c_6$ such that for $i\in\{1,2,\ldots,6\}$ and $r_i\in\ro^+$, $h(t_i)-h(s_i)= e_ir_i$. This gives us a system of $6$ linear equations in $4$ unknowns as follows:
\begin{eqnarray}
c_6(t_1^6-s_1^6)+c_3(t_1^3-s_1^3)+c_2(t_1^2-s_1^2)+c_1(t_1-s_1)&=& e_1r_1\label{eq2.6}\\
c_6(t_2^6-s_2^6)+c_3(t_2^3-s_2^3)+c_2(t_2^2-s_2^2)+c_1(t_2-s_2)&=& e_2r_2\\
c_6(t_3^6-s_3^6)+c_3(t_3^3-s_3^3)+c_2(t_3^2-s_3^2)+c_1(t_3-s_3)&=& e_3r_3\\
c_6(t_4^6-s_4^6)+c_3(t_4^3-s_4^3)+c_2(t_4^2-s_4^2)+c_1(t_4-s_4)&=& e_4r_4\\
c_6(t_5^6-s_5^6)+c_3(t_5^3-s_5^3)+c_2(t_5^2-s_5^2)+c_1(t_5-s_5)&=& e_5r_5\\
c_6(t_6^6-s_6^6)+c_3(t_6^3-s_6^3)+c_2(t_6^2-s_6^2)+c_1(t_6-s_6)&=& e_6r_6\label{eq2.11}
\end{eqnarray}
\noindent The rank of the coefficient matrix 
\[
B=
\left[ {\begin{array}{cccc}
t_1^6-s_1^6 & t_1^3-s_1^3 & t_1^2-s_1^2 & t_1-s_1 \\
t_2^6-s_2^6 & t_2^3-s_2^3 & t_2^2-s_2^2 & t_2-s_2 \\
t_3^6-s_3^6 & t_3^3-s_3^3 & t_3^2-s_3^2 & t_3-s_3 \\
t_4^6-s_4^6 & t_4^3-s_4^3 & t_4^2-s_4^2 & t_4-s_4 \\
t_5^6-s_5^6 & t_5^3-s_5^3 & t_5^2-s_5^2 & t_5-s_5 \\
t_6^6-s_6^6 & t_6^3-s_6^3 & t_6^2-s_6^2 & t_6-s_6 \\
\end{array} } \right]
 \]\vw
\noindent of the system of linear equations \ref{eq2.6} - \ref{eq2.11} is at most $4$. The system has a solution if and only if the rank of $B$ is equal to the rank of the augmented matrix $\tilde{B_e}$ (where the subscript $e$ denotes the dependence of $\tilde{B_e}$ on the pattern $e$). Thus, system of linear equations \ref{eq2.6} - \ref{eq2.11} has no solution if $\tilde{B_e}$ has full rank (i.e. $rank(\tilde{B_e})=5$).

Note that for each $i$, $t_i$ and $s_i$ are analytic functions of the coefficients $a_k$'s of $f$ and the coefficients $b_k$'s of $g$. Thus $\tilde{B_e}$ is a non-constant analytic function of $a_k$'s, $b_k$'s and $r_k$'s. Hence the set $$V_e=\{\hf(\hp a_0,\ldots,a_4,b_0,\ldots,b_5, r_1,\ldots,r_6\hp)\in U_e\times(\mathbb{R}^+)^6\mid\tilde{B_e}\hskip1.3mm \mbox{has full rank}\hf\}$$ is an open and a dense subset of $U_e\times(\mathbb{R}^+)^6$. It is clear that, for any choice of an element in $V_e$, the system of linear equations \ref{eq2.6} - \ref{eq2.11} has no solution.\vo
 
Now consider the disjoint union $V_2=\bigsqcup_eV_e$ which is clearly an open and a dense subset of the disjoint union $U_3=\bigsqcup_eU_e\times(\mathbb{R}^+)^6$, where both the unions are taken over all the patterns $e$. Note that $U_2\times(\mathbb{R}^+)^6\subseteq U_3$. It is easy to see that, for any choice of an element in $V_2$, the corresponding system of linear equations has no solution. Hence, for a generic choice of a regular projection $t\mapsto\fgt$ with $6$ crossings and having $\deg(f)=4$ and $\deg(g)=5$, there does not exist a polynomial $h$ of degree $6$ such that $t\mapsto\fght$ represents the knot $5_2$.
\end{proof} 
 
\begin{conjecture}\label{coj1}
 The knot $5_2$ can not be realized in the space $\ps$.
\end{conjecture}

It is conjectured that the only 3-superbridge knots are $3_1$ and $4_1$ \cite{jw}. Once it is proved, then the conjecture \ref{coj1} will follow trivially. Similarly, we conjecture that the knots $6_1, 6_2$ and $6_3$ can not be realized in $\ps$.
 
%%%%%%%%%%%%%%%%%%%%%%%%%%%%%%%%%%%%%%%%%%%%%%%%%%%%%%%%%%%%%%%%%%%%%%%%%%%%%%%%%%%%%%%
\section{Spaces of polynomial knots}\label{sec3}

\subsection{The spaces $\pd$ and $\pdt$}\label{sec3.1}

\noindent\hskip6mm For a positive integer $d$, Derfee and O'Shea \cite{do} have discussed the topology (it is inherited from $\ro^{3d+3}$) of the space $\kd$ of all polynomial knots with degree $d$ (that is, the degrees of the component polynomials are at most $d$ and at least one of the degree is $d$). Let $\spk$ denote the set of all polynomial knots. We can write $$\spk=\bigcup_{d\geq1}\kd\hf.$$ This set can be given the inductive limit topology; that is, a set $U\subseteq\spk$ is open in $\spk$ if and only if the set $U\cap\kd$ is open in $\kd$ for all $d\geq1$. Thus, we have the space $\spk$ of all polynomial knots.

\begin{definition}\label{def3.1}
Two polynomial knots $\phi,\psi\in\spk$ are said to be {\sf polynomially isotopic} if there exists a one parameter family $\big\{\hf\Phi_s\in\spk\mid\! s\in\zo\hf\big\}$ of polynomial knots such that $\Phi_0=\phi$ and $\Phi_1=\psi$.
\end{definition}

In the definition above, one has to note that a map $\Phi:\zo\times\ro\rightarrow\rt$ defined by $s\mapsto\Phi_s$ is continuous. It has been proved that if two polynomial knots are topologically equivalent as long knots  then they are polynomially isotopic \cite{rs1}. Any polynomial isotopy within the space $\kd$ for a fixed $d$ gives rise to a smooth path inside $K_d$. However, two polynomial knots belonging to two different spaces $\kd$ and $\mathcal{K}_{d'}$ do not belong to the same path component of the space $\spk$ even though they are polynomially isotopic or topologically equivalent. This motivates us to define a equivalence based on knots belonging to the same path component inside a space of polynomial knots with a given condition on the degrees of the component polynomials. Recall that $\pd$ be the space of all polynomial knots $\sigma=\fgh$ with $\fghio$ and $\pdt$ be the space of all polynomial knots $\varsigma=\fghp$ with $\deg(f')=d-2,\,\deg(g')=d-1$\hskip1.5mm and\hskip1.5mm$\deg(h')=d$.

\begin{definition}
Two polynomial knots are said to be {\sf path equivalent} in $\pd$ if they belong to the same path component of the space $\pd$.
\end{definition}

Similarly, the path equivalence can be defined for the spaces $\pdt$ and $\kd$. Also, it can be defined for the space $\spk$ of all polynomial knots. Using advanced techniques of differential topology Durfee and O'Shea gave a proof (see \cite{do}, proposition 9) for the following fact:

\begin{proposition}\label{thm8}
If two polynomial knots are path equivalent in $\kd$, then they are topologically equivalent.
\end{proposition}

\begin{corollary}\label{cor0}
If two polynomial knots are path equivalent in $\spk$, then they are topologically equivalent.
\end{corollary}  

\begin{proof}
Since all the sets $\kd$, for $d\geq1$, are open and closed in $\spk$, so any path in $\spk$ joining two polynomial knots lies wholly in $\kd$ for some $d\geq1$. Thus, if two polynomial knots  are path equivalent in $\spk$, then they are so in $\kd$ for some $d\geq1$ and hence they are topologically equivalent. \end{proof}

\begin{corollary}\label{cor1}
If two polynomial knots are path equivalent in $\pdt$, then they are topologically equivalent. 
\end{corollary}  

\begin{proof}
Let $\phi$ and $\psi$ be polynomial knots belonging to a same path component of the space $\pdt$. Since $\pdt\subset\kd$, so $\phi$ and $\psi$ are members of $\kd$ belonging to its same path component. Therefore, by proposition \ref{thm8}, they are topologically equivalent.    
\end{proof}

\begin{theorem}\label{thm9}
Suppose $\phi=\fgh\in\pdt$ be a polynomial representation of a classical tame knot $\kappa$, then $\phi$ and its mirror image $\psi=\fgmh$ belong to the different path components of the space $\pdt$.
\end{theorem}

\begin{proof}
Suppose contrary that $\phi=\fgh$ and $\psi=\fgmh$ belong to the same path component of $\pdt$ and let $\Phi:\zo\to\pdt$ be a path from $\phi$ to $\psi$. For $s\in\zo$, let $\Phi_s=\Phi(s)$ and let it be given by
\begin{eqnarray*}
\Phi_s(t)&=&\big(\ho\alpha_{d-2}(s)t^{d-2}+\cdots+\alpha_1(s)t+\alpha_0(s),\hs\beta_{d-1}(s)t^{d-1}+\cdots+\beta_1(s)t+\beta_0(s),\\ & & \hskip3mm\gamma_d(s)t^d+\cdots+\gamma_1(s)t+\gamma_0(s)\ho\big)
\end{eqnarray*}
\noindent for $t\in\ro$. The maps $\alpha_i$'s, $\beta_i$'s and $\gamma_i$'s are continuous. Let $f,\, g$ and $h$ be given by
\begin{eqnarray*}
f(t)&=&\at\hf,\\ g(t)&=&\bt\hw \mbox{and}\\ h(t)&=&\ct\hf.
\end{eqnarray*}
\noindent Since $\Phi_0=\phi=\fgh$ and $\Phi_1=\psi=\fgmh$, so for each $i$, we have $\alpha_i(0)=\alpha_i(1)=a_i,\hs\beta_i(0)=\beta_i(1)=b_i,\hs\gamma_i(0)=c_i$ and $\gamma_i(1)=-c_i$\hf. In particular, $\gamma_d(0)=c_d$ and $\gamma_d(1)=-c_d$. Since $c_d\neq0$ and $\gamma_d$ is a continuous function, so by the intermediate value theorem, $\gamma_d(s_0)=0$ for some $s_0\in\ozo$. Therefore, the third component of $\Phi_{s_0}$ has degree less than $d$ and thus it does not a belong to the space $\pdt$. This is a contradiction to the fact that $\Phi$ is a path in $\pdt$.
\end{proof}

\begin{corollary}\label{cor2}
Let $\varphi=\fgh\in\pdt$ ($d\geq3$) be a polynomial representation of a tame knot $\kappa:\so\to\st$. Then we have the following:\vo
\noindent\hskip2.25mm i) If $\kappa$ is acheiral, then it corresponds to at least eight path components of $\pdt$.\vo  
\noindent\hskip1mm ii) If $\kappa$ is cheiral, then each $\kappa$ and $\kappa^*$ corresponds to at least four path components\vo\noindent\hskip6.6mm of the space $\pdt$.
\end{corollary}

\begin{proof}
Using the argument as used in the proof of theorem \ref{thm9}, it is easy to see that $\varphi_1=\fgh,\,\varphi_2=(\hp -f,\hp -g,\hp h\hp),\,\varphi_3=(\hp -f,\hp g,\hp -h\hp),\,\varphi_4=(\hp f,\hp -g,\hp -h\hp),\,\varphi_5=\mifgh,\,\varphi_6=\fmgh,\,\varphi_7=\fgmh$ and $\varphi_8=\minusfgh$ belong to eight distinct path components of $\pdt$. If $\kappa$ is acheiral, then all the knots $\varphi_1, \varphi_2,\dots,\varphi_8$ represent it. If $\kappa$ is cheiral, then the knots $\varphi_1, \varphi_2, \varphi_3$ and $\varphi_4$ represent $\kappa$ and the knots $\varphi_5, \varphi_6, \varphi_7$ and $\varphi_8$ represent the mirror image $\kappa^*$ of $\kappa$.
\end{proof}

\begin{remark}\label{rem6}
For $d\geq3$ and $\phi=\fgh\in\pdt$, there are eight distinct path components of the space $\pdt$ each of which contains exactly one of the knot $\phi_e=(e_1 f,e_2\hp g,e_3\hp h)$ for $e=(e_1,e_2,e_3)$ in $\{-1,1\}^3$. This shows that the total number of path components of the space $\pdt$, for $d\geq3$, are in multiple of eight.
\end{remark}

\begin{remark}\label{rem7}
If $n$ distinct knot-types (up to mirror images) are represented in $\pdt$, then it has at least $8n$ distinct path components.
\end{remark}

%%%%%%%%%%%%%%%%%%%%%%%%%%%%%%%%%%%%%%%%%%%%%%%%%%%%%%%%%%%%%%%%%%%%%%%%%%%%%%%%%%%%%%%
\subsection{ The spaces $\pd$ and $\pdt$ for $d\leq 4$}\label{sec3.2}

\noindent\hskip6mm For a polynomial map $\phi\in\ad$, let us denote its first, second and third components respectively by $f_\phi,\, g_\phi$\, and\, $h_\phi$. Also, for\, $i=0,1,\ldots, d$,\, we denote the coefficients of $t^i$ in the polynomials $f_\phi,\, g_\phi$ and $h_\phi$ by $a_{i\phi},\,b_{i\phi}$\, and $c_{i\phi}$ respectively. Sometimes we use letters $\varphi,\psi, \tau,\sigma,\varsigma$ and $\omega$ to denote the elements of $\ad$. In such cases, the corresponding components and their coefficients will be denoted using corresponding subscripts. For example, for $\sigma\in\mathcal{A}_4$, its second component will be denoted by $g_\sigma$ and $b_{2\sigma}$ will denote the coefficient of $t^2$ in the polynomial $g_\sigma$.

\begin{proposition}\label{thm10}
The space $\pw$  is open in $\mathcal{A}_2$ and it has exactly four path components.
\end{proposition}

\begin{proof}
Note that $\pw=\big\{\hs\phi\in\mathcal{A}_2\mid b_{1\phi}\hp c_{2\phi}\neq0\hs\big\}$\hp. By the identification of the space $\mathcal{A}_2$ with $\ro^6$, it is easy to see that the space $\pw$ is naturally homeomorphic to the open subset $\big\{\hs(\hp a_0,b_0,b_1,c_0,c_1,c_2 \hp) \in\ro^6\mid b_1c_2\neq0 \hs\big\}$ of the Euclidean space $\ro^6$. Therefore, the set $\pw$ is open in $\mathcal{A}_2$ and it has four path components as follows:\vo

\noindent$\mathcal{P}_{21}=\big\{\hs\phi\in\mathcal{A}_2\mid b_{1\phi}>0$ \& $c_{2\phi}>0\hs\big\}$\hp, $\mathcal{P}_{22}=\big\{\hs\phi\in\mathcal{A}_2\mid b_{1\phi}>0$ \& $c_{2\phi}<0\hs\big\}$\hp,\vo
\noindent$\mathcal{P}_{23}=\big\{\hs\phi\in\mathcal{A}_2\mid b_{1\phi}<0$ \& $c_{2\phi}>0\hs\big\}$\hp,
$\mathcal{P}_{24}=\big\{\hs\phi\in\mathcal{A}_2\mid b_{1\phi}<0$ \& $c_{2\phi}<0\hs\big\}$\hp.
\end{proof}

\begin{remark}\label{rem8}
It is easy to see that the sets $\pw$ and $\pwt$ are equal and thus proposition \ref{thm10} is also true if\, $\pw$ is replaced by $\pwt$.
\end{remark}

\begin{lemma}\label{thm11}
Let $X$ be a topological space and let $\mathfrak{F}$ be an arbitrary covering (with at least two members) for $X$ by its non-empty subsets. Let $U$ and $V$ be two distinct members of the covering\, $\mathfrak{F}$. Then for $X$ to be path connected, it is enough to satisfy the following conditions:\vo
\noindent\hskip2.25mm i) For any $u\in U$ and any $v\in V$ there is a path from $u$ to $v$.\vo
\noindent\hskip1mm ii) For any $W\in\mathfrak{F}$ and any $x\in W$ there exists an element $y\in U\cup V$ such that\vo\noindent\hskip6.6mm there is a path from $x$ to $y$.
\end{lemma}

If the cover $\mathfrak{F}$ contains only two non-empty distinct subsets $U$ and $V$, then for $X$ to be path connected it is sufficient to satisfy the first condition of the lemma.

\begin{theorem}\label{thm12}
The space $\pt$ is path connected.
\end{theorem}

\begin{proof}
We consider the following sets:\vo 
\hskip10mm$U_1=\{\ho\varphi\in\mathcal{A}_3\mid\varphi$ has degree sequence $(\hp0,1,2\hp)\hs\}$\hs,\vo
\hskip10mm$U_2=\{\ho\varphi\in\mathcal{A}_3\mid\varphi$ has degree sequence $(\hp0,1,3\hp)\hs\}$\hs,\vo
\hskip10mm$U_3=\{\ho\varphi\in\mathcal{A}_3\mid\varphi$ has degree sequence $(\hp0,2,3\hp)\hs\}\cap\mathcal{P}_3$\hw and\vo
\hskip10mm$U_4=\{\ho\varphi\in\mathcal{A}_3\mid\varphi$ has degree sequence $(\hp1,2,3\hp)\hs\}=\ptt$\hs.\vw

\noindent It is easy to note that these sets are pairwise disjoint and their union is exactly equal to $\mathcal{P}_3$. To prove the theorem, we proceed as follows:\vo

$i)$ Let $\phi\in U_1$ and $\psi\in U_4$ be arbitrary elements. Let $\Phi:\zo\rightarrow\mathcal{A}_3$ be given by $\Phi(s)=\Phi_s$ for $s\in\zo$, where$$\Phi_s(t)=(1-s)\hs\phi(t)+s\hs\psi(t)$$\noindent for $t\in\ro$. It is clear that $\Phi_0=\phi$ and $\Phi_1=\psi$. For $s\in\ozo$, $\Phi_s$ has degree sequence $(\hs1,2,3\hs)$ and thus it is an element of $U_4$. This shows that $\Phi$ is a path in $\mathcal{P}_3$ from $\phi$ to $\psi$.\vo

$ii)$ Suppose $\tau=\fgh$ be an arbitrary element of $U_2\cup U_3$. Let us choose an element $e_2\in\{-1,1\}$ such that $e_2\hp b_{2\tau}\geq0$. Let $\Psi:\zo\rightarrow\mathcal{A}_3$ be given by $\Psi(s)=\Psi_s$ for $s\in\zo$, where 
$$\Psi_s(t)=\left(\hs (1-s)f(t)+s\hf t,\hs (1-s)g(t)+e_2\hp s\hf t^2,\hs h(t)\hs\right)$$ 

\noindent for $t\in\ro$. Note that $\Psi_0=\tau$ and $\Psi_1=\sigma$, where $\sigma$ is given by $\sigma(t)=\big(\hp t, e_2\hp t^2, h(t)\hp\big)$ for $t\in\ro$. It is easy to see that $\sigma\in U_4$ and $\Psi_s\in U_4$ for all $s\in\ozo$. Therefore, we have a path in $\mathcal{P}_3$ from $\tau$ to $\sigma$.\vo

The first and second parts above satisfy respectively the first and second assumptions of lemma \ref{thm11}, so by this lemma, the space $\mathcal{P}_3$ is path connected. 
\end{proof}

\begin{proposition}\label{thm17}
The space $\ptt$ has eight path components.
\end{proposition}
 
\begin{proof}  
For $e=(e_1,e_2,e_3)$ in $\{-1,1\}^3$, consider the following set:\vo
 
\hskip1mm$\mathcal{\tilde{P}}_{3e}=\big\{\hs\phi\in\mathcal{A}_3\,\mid\, e_1\hp a_{1\phi}>0,\; e_2\hp b_{2\phi}>0\hw \text{and}\hw e_3\hp c_{3\phi}>0\hs\big\}$\vo
\hskip7.9mm$\cong\big\{\hp(\hp a_0,a_1,b_0,b_1,b_2,c_0,c_1,c_2,c_3\hp) \in\ro^9\mid e_1\hp a_1>0,\,e_2\hp b_2>0$ and $e_3\hp c_3>0\hp\big\}$,\vo

\noindent where \textquotedblleft\,$\cong$\,\textquotedblright\, denotes the natural homeomorphism of the spaces under the identification of $\mathcal{A}_3$ with $\ro^9$. It is easy to see that $\ptt=\bigcup_e\mathcal{\tilde{P}}_{3e}$. For any $e\in\{-1,1\}^3$, the set $\mathcal{\tilde{P}}_{3e}$ is path connected. Also, for $e\neq e'$, there is no path in $\ptt$ from an element of the set $\mathcal{\tilde{P}}_{3e}$ to an element of the set $\mathcal{\tilde{P}}_{3e'}$\hp. Thus, the sets $\mathcal{\tilde{P}}_{3e}$, for $e\in\{-1,1\}^3$, are nothing but the path components of the space $\ptt$. 
\end{proof}

\begin{proposition}\label{thm13}
A polynomial map $\phi\in\mathcal{A}_4$ given by $t\mapsto(\hs t^2+a\hf t\hs,\hs t^3+b\hf t,\hs t^4+c\hf t\hs)$ is an embedding if and only if\: $3\hf a^2+4\hf b>0$\, or\, $a^3+2\hf a\hf b+c\neq0$\hp.
\end{proposition}

\begin{proof}
Note that, a map $\phi\in\mathcal{A}_4$ given by $t\mapsto(\hs t^2+a\hf t\hs,\hs t^3+b\hf t,\hs t^4+c\hf t\hs)$ is an embedding $\Leftrightarrow$  $\phi(s)\neq\phi(t)$ for all $s,t\in\ro$ with $s\neq t$ and $\phi'(u)\neq0$ for all $u\in\ro$ $\Leftrightarrow$ the equations
\begin{eqnarray}
(s+t)+a&=& 0\hf,\label{eq1}\\
(s^2+s\hf t+t^2)+b&=& 0 \label{eq2}\hw\mbox{and}\\
(s^3+s^2\hf t+s\hf t^2+t^3)+c&=& 0\label{eq3}
\end{eqnarray}
\noindent does not have a common real solution. Using \ref{eq1} in \ref{eq2}, we get
\begin{eqnarray}
t^2+a\hf t+(a^2+b)&=& 0\hp.\label{eq4}
\end{eqnarray}

\noindent This quadratic equation has solutions $t=t_1$ and $t=t_2$, where
\begin{eqnarray*}
t_1=\dfrac{-a-\sqrt{-3\hf a^2-4\hf b}}{2}\hskip5mm \mbox{and} \hskip5mmt_2=\dfrac{-a+\sqrt{-3\hf a^2-4\hf b}}{2}\;.
\end{eqnarray*}
In order to prove the proposition, it is sufficient to check the following three statements (in fact, they are easy to check):\vo
1) $3\hf a^2+4\hf b>0$ $\Leftrightarrow$ the equation \ref{eq4} has no real solution $\Leftrightarrow$ the equations \ref{eq1} and \ref{eq2} does not have a common real solution.\vo 
2) $3\hf a^2+4\hf b=0$ and $a^3+2\hf a\hf b+c\neq0$ $\Leftrightarrow$  $3\hf a^2+4\hf b=0$ and $a^3-2\hf c\neq0$ $\Leftrightarrow$ $t=-a/2$ is the only solution of the equation \ref{eq4} and $a^3-2\hf c\neq0$ $\Leftrightarrow$ $(s,t)=(-a/2,-a/2)$ is the only common real solution of the equations \ref{eq1} and \ref{eq2}, but it is not a solution of the equation \ref{eq3}.\vo

3) $3\hf a^2+4\hf b<0$ and $a^3+2\hf a\hf b+c\neq0$ $\Leftrightarrow$ $t=t_1$ and $t=t_2$ are two distinct solutions of the equation \ref{eq4} and $a^3+2\hf a\hf b+c\neq0$ $\Leftrightarrow$ $(s,t)=(t_1,t_2)$ and $(s,t)=(t_2,t_1)$ are the only common real solutions of the equations \ref{eq1} and \ref{eq2}, but no one is a solution of the equation \ref{eq3}.
\end{proof}

\begin{corollary}\label{cor3}
For $e_1, e_2, e_3\in\{\hs-1, 1\hs\}$ and $a,b,c\in\ro$, a polynomial map $\tau\in\mathcal{A}_4$ given by $t\mapsto(\hs e_1\hf t^2+a\hf t\hs,\hs e_2\hf t^3+b\hf t,\hs e_3\hf t^4+c\hf t\hs)$ is an embedding if and only if\: $3 a^2+4 e_2 b>0$\: or\: $e_1 a^3+2 e_1 e_2\hp a\hp b+ e_3 c\neq0$\hp.
\end{corollary}

\begin{proof}
It is easy to note that $\tau$ is an embedding $\Leftrightarrow$ $\phi\in\mathcal{A}_4$ given by 
$$t\mapsto\left(\hs e_1(e_1\hf t^2+a\hf t)\hs,\hs e_2(e_2\hf t^3+b\hf t),\hs e_3(e_3\hf t^4+c\hf t)\hs\right)$$ \noindent is an embedding. The map $\phi$ can be written as $t\mapsto\left(\hs t^2+e_1\hp a\hf t,\hf t^3+e_2\hp b\hf t,\hf t^4+e_3\hp c\hf t\hs\right)$. Thus, using proposition \ref{thm13}, one can say that $\phi$ is an embedding $\Leftrightarrow$ $3 a^2+4 e_2 b>0$ or $e_1 a^3+2 e_1 e_2\hp a\hp b+ e_3 c\neq0$.
\end{proof}

\begin{proposition}\label{thm14} 
For $e_1, e_2, e_3\in\{\hs-1, 1\hs\}$ and $a,b,c\in\ro$, a polynomial knot $\varphi\in\pfrt$ given by $t\mapsto(\hs e_1\hf t^2+a\hf t,\hs e_2\hf t^3+b\hf t,\hs e_3\hf t^4+c\hf t\hs)$ is path equivalent (in $\pfr$) to at least one of the polynomial knot $t\mapsto(\hs0,\hs e_2\hf t,\hs e_3\hf t^2\hs)$\, or\, $t\mapsto(\hs0,\hs -e_2\hf t,\hs -2\hf e_3\hf t^2\hs)$.
\end{proposition}

\begin{proof}
By corollary \ref{cor3}, for $\varphi\in\pfr$ given by $t\mapsto(\hs e_1\hf t^2+a\hf t,\hs e_2\hf t^3+b\hf t,\hs e_3\hf t^4+c\hf t\hs)$, we have\, $3 a^2+4 e_2 b>0$\, or\, $e_1\hf a^3+2\hf e_1\hf e_2\hf a\hf b+e_3\hf c\neq0$. We now consider the following two cases:\vo

$i)$ If $3 a^2+4 e_2 b>0$\hp: Let $\Phi:\zo\rightarrow\mathcal{A}_4$ be a map which is given by $\Phi(s)=\Phi_s$ for $s\in\zo$, where$$\Phi_s(t)=\left(\hs (1-s)(e_1\hf t^2+a\hf t),\hs(1-s)(e_2\hf t^3+b\hf t)+e_2\hf s\hs t,\hs (1-s)(e_3\hf t^4+c\hf t)+e_3\hs s\hs t^2\hs\right)$$\noindent for $t\in\ro$. For $s\in\ozo$, the polynomial map $\Phi_s$ is an embedding $\Leftrightarrow$ the polynomial map $\frac{1}{1-s}\Phi_s\in\mathcal{A}_4$ given by $$t\mapsto\bigg(\hs e_1\hf t^2+a\hf t,\hs e_2\hf t^3+\Big(b+\dfrac{e_2\hf s}{1-s}\Big)t,\hs e_3\hf t^4+\dfrac{e_3\hf s}{1-s}\hf t^2+c\hf t\hs\bigg)$$ is an embedding $\Leftrightarrow$ the polynomial map $\Psi_s\in\mathcal{A}_4$ given by $$t\mapsto\bigg(\hs e_1\hf t^2+a\hf t,\hs e_2\hf t^3+\Big(b+\dfrac{e_2\hf s}{1-s}\Big)t,\hs e_3\hf t^4+\Big( c-\dfrac{e_1\hf e_3\hf a\hf s}{1-s}\Big) t\hs\bigg)$$ \noindent is an embedding. For $s\in\ozo$, since $3\hf a^2+4\hf e_2\big(b+\frac{e_2\hf s}{1-s}\big)=3\hf a^2+4\hf e_2\hf b+\frac{4\hf s}{1-s}>0$, so by corollary \ref{cor3}, the map $\Psi_s$ is an embedding and hence so is the map $\Phi_s$. Note that $\Phi_0=\varphi$ and $\Phi_1=\tau$, where $\tau$ is given by $t\mapsto(\hs0,\hs e_2\hf t,\hs e_3\hf t^2\hs)$. Thus, there is path in $\pfr$ from $\varphi$ to $\tau$.\vo

$ii)$ If $e_1\hf a^3+2\hf e_1\hf e_2\hf a\hf b+e_3\hf c\neq0$\hp: Let $\varUpsilon:\zo\rightarrow\mathcal{A}_4$ be a map which given by $\varUpsilon(s)=\varUpsilon_s$ for $s\in\zo$, where
$$\varUpsilon_s(t)=\left((1-s)(e_1\hf t^2+a\hf t),\hf(1-s)(e_2\hf t^3+b\hf t)-e_2\hf s\hf t,\hf (1-s)(e_3\hf t^4+c\hf t)-2\hf e_3\hf s\hf t^2\right)$$ 
\noindent for $t\in\ro$. For $s\in\ozo$, the polynomial map $\varUpsilon_s$ is an embedding $\Leftrightarrow$ the polynomial map $\frac{1}{1-s}\varUpsilon_s\in\mathcal{A}_4$ given by
$$t\mapsto\bigg(\hs e_1\hf t^2+a\hf t,\hs e_2\hf t^3+\Big(b-\dfrac{e_2\hf s}{1-s}\Big)t,\hs e_3\hf t^4-\dfrac{2\hf e_3\hf s}{1-s}\hf t^2+c\hf t\hs\bigg)$$ is an embedding $\Leftrightarrow$ the polynomial map $\varGamma_s\in\mathcal{A}_4$ given by $$t\mapsto\bigg(\hs e_1\hf t^2+a\hf t,\hs e_2\hf t^3+\Big(b-\dfrac{e_2\hf s}{1-s}\Big)t,\hs e_3\hf t^4+\Big( c+\dfrac{2\hf e_1\hf e_3\hf a\hf s}{1-s}\Big) t\hs\bigg)$$ \noindent is an embedding. For $s\in\ozo$, since $$e_1\hf a^3+2\hf e_1\hf e_2\hf a\Big(b-\dfrac{e_2\hf s}{1-s}\Big)+e_3\Big( c+\dfrac{2\hf e_1\hf e_3\hf a\hf s}{1-s}\Big) =e_1\hf a^3+2\hf e_1\hf e_2\hf a\hf b+e_3\hf c\neq0\hp,$$ so by corollary \ref{cor3}, the map $\varGamma_s$ is an embedding and hence so is the map $\varUpsilon_s$. Note that $\varUpsilon_0=\varphi$ and $\varUpsilon_1=\sigma$, where $\sigma$ is given by $t\mapsto(\hs0,\hs -e_2\hf t,\hs -2\hf e_3\hf t^2\hs)$. This shows that the map $\varUpsilon$ is a path in $\pfr$ from $\varphi$ to $\sigma$.
\end{proof}

\begin{proposition}\label{thm15}
For any $\phi\in\pfrt$, there exist $e_1, e_2, e_3\in\{\hs-1, 1\hs\}$ and $a,b,c\in\ro$  such that  a polynomial map $\psi$ given by $t\mapsto(\hs e_1\hf t^2+a\hf t\hs,\hs e_2\hf t^3+b\hf t,\hs e_3\hf t^4+c\hf t\hs)$ is an embedding which is path equivalent (in $\pfrt$) to the polynomial knot $\phi$.
\end{proposition}   

\begin{proof} We prove this proposition in the following steps:\vo
$i)$ Let $\phi\in\pfrt$ be a polynomial knot and let it be given by $$t\mapsto\big(\hs a_2 t^2+a_1t+a_0,\hs b_3t^3+b_2t^2+b_1t+b_0,\hs c_4t^4+c_3t^3+c_2t^2+c_1t+c_0\hs\big).$$ \noindent We take a map  $\Phi:\zo\rightarrow\mathcal{A}_4$ which is given by $\Phi(s)=\Phi_s$ for $s\in\zo$, where\vw
\hskip10mm$\Phi_s(t)=\big(\hs a_2 t^2+a_1t+(1-s)\hf a_0,\hs b_3t^3+b_2t^2+b_1t+(1-s)\hf b_0,\hs c_4t^4+c_3t^3+$\vo\hskip27.4mm $c_2t^2+c_1t+(1-s)\hf c_0\hs\big)$\vw
\noindent for $t\in\ro$. It is easy to check that $\Phi_s\in\pfrt$ for all $s\in\zo$. Let $\tau=\fgh$, where $f(t)=a_2 t^2+a_1t,\hs g(t)=b_3t^3+b_2t^2+b_1t$ and $h(t)=c_4t^4+c_3t^3+c_2t^2+c_1t$ for $t\in\ro$. Clearly $\Phi_0=\phi$ and $\Phi_1=\tau$. This gives a path in $\pfrt$ from $\phi$ to $\tau$.\vo

$ii)$ Note that $a_2, b_3$ and $c_4$ all are nonzero. Let $\Psi:\zo\rightarrow\mathcal{A}_4$ be a map which is given by $\Psi(s)=\Psi_s$ for $s\in\zo$, where $$\Psi_s(t)=\Big( \hs f(t),\hs g(t)-\frac{b_2}{a_2}s\hf f(t),\hs h(t)+\frac{b_2c_3-b_3c_2}{a_2b_3}s\hf f(t)-\frac{c_3}{b_3}s\hf g(t)\hs\Big)$$ \noindent for $t\in\ro$. Note that $\Psi_s\in\pfrt$ for all $s\in\zo$. Let $\sigma=(\hp f, g_1, h_1\hp)$, where
\begin{align*}
g_1(t)&=b_3t^3+b_{11}t=b_3t^3+\Big(b_1-\frac{a_1\hf b_2}{a_2}\Big)\hp t\hw \mbox{and}\\
h_1(t)&=c_4t^4+c_{11}t=c_4t^4+\Big(c_1+\frac{a_1b_2c_3-a_1b_3c_2}{a_2b_3}-\frac{b_1c_3}{b_3}\Big)\hp t
\end{align*}
\noindent for $t\in\ro$. It is easy to check that $\Psi_0=\tau$ and $\Psi_1=\sigma$. So we have a path in $\pfrt$ from $\tau$ to $\sigma$.\vo

$iii)$ Let $\varUpsilon:\zo\rightarrow\mathcal{A}_4$ be given by $\varUpsilon(s)=\varUpsilon_s$ for $s\in\zo$, where
$$\varUpsilon_s(t)=\bigg(\hs\Big( 1-s+\dfrac{s}{|a_2|}\Big)f(t),\hs \Big( 1-s+\dfrac{s}{|b_3|}\Big) g_1(t),\hs\Big( 1-s+\dfrac{s}{|c_4|}\Big) h_1(t)\hs\bigg)$$ 
\noindent for $t\in\ro$. For $s\in\zo$, since $1-s+\frac{s}{|a_2|}>0,\ho 1-s+\frac{s}{|b_3|}>0$ and $1-s+\frac{s}{|c_4|}>0$, so $\varUpsilon_s$ is a polynomial knot in $\pfrt$. Let $\psi=(\hp f_2, g_2, h_2\hp)$, where
\begin{eqnarray*}
f_2(t)&=& e_1\hf t^2+a\hf t=\dfrac{a_2t^2+a_1t}{|a_2|}\,,\\
g_2(t)&=& e_2\hf t^2+b\hf t=\dfrac{b_3t^3+b_{11}t}{|b_3|}\hw \mbox{and}\\
h_2(t)&=& e_3\hf t^2+c\hf t=\dfrac{c_4t^4+c_{11}t}{|c_4|}
\end{eqnarray*}
\noindent for $t\in\ro$. It is easy to note that $\varUpsilon_0=\sigma$ and $\varUpsilon_1=\psi$. This shows that the map $\varUpsilon$ is a path in $\pfrt$ from $\sigma$ to $\psi$.
\end{proof}

The following corollary follows trivially from propositions \ref{thm14} and \ref{thm15}.

\begin{corollary}\label{cor4}
Any polynomial knot $\phi\in\pfrt$ is path equivalent (in $\pfr$) to a polynomial knot $\psi\in\pfr$ having degree sequence $(\hp0,1,2\hp)$.
\end{corollary}
 
\begin{theorem}\label{thm16}
The space $\pfr$ is path connected.
\end{theorem}

\begin{proof}
Consider the following sets:\vw 
\hskip10mm$V_1=\{\hf\phi\in\mathcal{A}_4\mid\phi$ has degree sequence $(\hp0,1,2\hp)\hf\}$\hs,\vo
\hskip10mm$V_2=\{\hf\phi\in\mathcal{A}_4\mid\phi$ has degree sequence $(\hp0,1,3\hp)\hf\}$\hs,\vo
\hskip10mm$V_3=\{\hf\phi\in\mathcal{A}_4\mid\phi$ has degree sequence $(\hp0,1,4\hp)\hf\}$\hs,\vo
\hskip10mm$V_4=\{\hf\phi\in\mathcal{A}_4\mid\phi$ has degree sequence $(\hp0,2,3\hp)\hf\}\cap\pfr$\hs,\vo
\hskip10mm$V_5=\{\hf\phi\in\mathcal{A}_4\mid\phi$ has degree sequence $(\hp0,2,4\hp)\hf\}\cap\pfr$\hs,\vo
\hskip10mm$V_6=\{\hf\phi\in\mathcal{A}_4\mid\phi$ has degree sequence $(\hp0,3,4\hp)\hf\}\cap\pfr$\hs,\vo
\hskip10mm$V_7=\{\hf\phi\in\mathcal{A}_4\mid\phi$ has degree sequence $(\hp1,2,3\hp)\hf\}$\hs,\vo
\hskip10mm$V_8=\{\hf\phi\in\mathcal{A}_4\mid\phi$ has degree sequence $(\hp1,2,4\hp)\hf\}$\hs,\vo
\hskip10mm$V_9=\{\hf\phi\in\mathcal{A}_4\mid\phi$ has degree sequence $(\hp1,3,4\hp)\hf\}$\hw and\vo
\hskip8.6mm$V_{10}=\{\hf\phi\in\mathcal{A}_4\mid\phi$ has degree sequence $(\hp2,3,4\hp)\hf\}\cap\pfr=\pfrt$\hp.\vw

\noindent Note that these sets are pairwise disjoint and their union is exactly equal to $\pfr$. Using the similar argument as used in the proof of theorem \ref{thm12}, one can show that there is a path in $\pfr$ from an arbitrary element of $V_1$ to an arbitrary element of $V_9$. Also by the similar argument as used in the second part of the proof of theorem \ref{thm12}, it is easy to produce a path from an arbitrary element of $\bigcup_{i=2}^8 V_i$ to an element of $V_9$. Using corollary \ref{cor4}, one has a path from an arbitrary element of $V_{10}$ to an element of $V_1$. This satisfies both the assumptions of lemma \ref{thm11} and hence $\pfr$ is path connected. 
\end{proof}

We have proved that the spaces $\pt$ and $\pfr$ are path connected, so in general, we would like to conjecture the following:

\begin{conjecture}\label{coj2}
The space $\pd$, for $d\geq3$, is path connected. 
\end{conjecture}

\begin{proposition}\label{thm18}
For any fixed element $e=(e_1,e_2,e_3)$ in $\{-1,1\}^3$, the space $\mathcal{N}_{4e}=\big\{\hs \phi\in\mathcal{\tilde{P}}_4\mid\; t\xrightarrow{\phi}(\hp e_1t^2+a_\phi\hp t,e_2t^3+b_\phi\hp t,e_3t^4+c_\phi\hp t\hp)$\hw for some\hw $a_\phi,b_\phi,c_\phi\in\ro\hs\big\}$\\ is path connected.
\end{proposition}

\begin{proof}
Suppose $e=(e_1,e_2,e_3)$ be an arbitrary element in the set $\{-1,1\}^3$. By corollary \ref{cor3}, an element $\phi\in\mathcal{A}_4$ given by $t\mapsto(\hp e_1t^2+a_\phi\hp t,e_2t^3+b_\phi\hp t,e_3t^4+c_\phi\hp t\hp)$ is an embedding if and only if $3a_\phi^2+4e_2b_\phi>0$\, or\, $e_1a_\phi^3+2e_1e_2a_\phi b_\phi+e_3c_\phi\neq0$. Therefore, it is easy to see that the set $\mathcal{N}_{4e}$ is the union of the following sets:\vw

$\hskip20mm \mathcal{N}_{4e}^1=\big\{\hf\phi\in\mathcal{N}_{4e}\mid 3a_\phi^2+4e_2b_\phi>0\hf\big\}$\hf,\vo

$\hskip20mm\mathcal{N}_{4e}^2=\big\{\hf\phi\in\mathcal{N}_{4e}\mid e_1a_\phi^3+2e_1e_2a_\phi b_\phi+e_3c_\phi>0\hf\big\}$\, and\vo

$\hskip20mm\mathcal{N}_{4e}^3=\big\{\hf\phi\in\mathcal{N}_{4e}\mid e_1a_\phi^3+2e_1e_2a_\phi b_\phi+e_3c_\phi<0\hf\big\}$\hf.\vw

\noindent We show that every element of the sets $\mathcal{N}_{4e}^1$, $\mathcal{N}_{4e}^2$ and $\mathcal{N}_{4e}^3$ is connected by a path in $\mathcal{N}_{4e}$ to a fixed element $\varphi_0\in\mathcal{N}_{4e}^1$ given by $t\mapsto\big(\hp e_1t^2,e_2t^3+e_2t,e_3t^4\hp\big)$.\vo

$i)$ Let $\varphi$ be an arbitrary element of the space $\mathcal{N}_{4e}^1$. Suppose $F:\zo\to\mathcal{A}_4$ be a map which is given by $F(s)=F_s$ for $s\in\zo$, where $$F_s(t)=\big(\hp e_1t^2+a_\varphi s\hp t,\,e_2t^3+(b_\varphi s^2+ e_2-e_2s^2)t,\,e_3t^4+c_\varphi s\hp t\hp\big)$$ for $t\in\ro$. For $s\in\ozo$, since $3(a_\varphi s)^2+4e_2(b_\varphi s^2+ e_2-e_2s^2)=s^2(3a_\varphi^2+4e_2b_\varphi)+4(1-s^2)>0$, so the map $F_s$ is an element of the space $\mathcal{N}_{4e}^1$.  Also, we have $F_0=\varphi_0$ and $F_1=\varphi$. This shows that $F$ is a path in $\mathcal{N}_{4e}^1$ from $\varphi_0$ to $\varphi$.\vo 

$ii)$ Let $\psi$ be an arbitrary element of the space $\mathcal{N}_{4e}^2$. We choose an element $\psi_0\in \mathcal{N}_{4e}^2$ given by $t\mapsto\big(\hp e_1t^2,e_2t^3,e_3t^4+e_3t\hp\big)$. Let a map $G:\zo\to\mathcal{A}_4$ be given by $G(s)=G_s$ for $s\in\zo$, where $$G_s(t)=\big(\hp e_1t^2+a_\psi s\hp t,\,e_2t^3+b_\psi s^2\hp t,\,e_3t^4+(c_\psi\hp s^3+ e_3-e_3s^3)\hp t\hp\big)$$ for $t\in\ro$. For $s\in\ozo$, the map $G_s$ is an element of the space $\mathcal{N}_{4e}^2$, because $e_1(a_\psi s)^3+2e_1e_2(a_\psi s)(b_\psi s^2)+e_3(c_\psi s^3+e_3-e_3s^3)=s^3(e_1a_\psi^3+2e_1e_2a_\psi b_\psi+e_3c_\psi)+1-s^3>0$.  Also, since $G_0=\psi_0$ and $G_1=\psi$, so the map $G$ is a path in $\mathcal{N}_{4e}^2$ from $\psi_0$ to $\psi$. Now take a map $\varPhi:\zo\to\mathcal{A}_4$ given by $\varPhi(s)=\varPhi_s$ for $s\in\zo$, where $$\varPhi_s(t)=\big(\hp e_1t^2,e_2t^3+e_2(1-s)t,e_3t^4+e_3 s\hp t\hp\big)$$ for $t\in\ro$. For $s\in\ozo$, the map $\varPhi_s$ is an element of the space $\mathcal{N}_{4e}^1$. Note that $\varPhi_0=\varphi_0$ and $\varPhi_1=\psi_0$. This shows that $\varPhi$ is a path in $\mathcal{N}_{4e}$ from $\varphi_0$ to $\psi_0$. The maps $G$ and $\varPhi$ both together gives a path in $\mathcal{N}_{4e}$ from $\varphi_0$ to $\psi$.\vo 

$iii)$ Let $\sigma$ be an arbitrary element of the space $\mathcal{N}_{4e}^3$. Take an element $\sigma_0\in \mathcal{N}_{4e}^3$ given by $t\mapsto\big(\hp e_1t^2,e_2t^3,e_3t^4-e_3t\hp\big)$. Let $H:\zo\to\mathcal{A}_4$ be a map which is given by $H(s)=H_s$ for $s\in\zo$, where $$H_s(t)=\big(\hp e_1t^2+a_\sigma s\hp t,\,e_2t^3+b_\sigma s^2\hp t,\,e_3t^4+(c_\sigma s^3-e_3+e_3s^3) t\hp\big)$$ for $t\in\ro$. For $s\in\ozo$, since $e_1(a_\sigma s)^3+2e_1e_2(a_\sigma s)(b_\sigma s^2)+e_3(c_\sigma s^3-e_3+e_3s^3)=s^3(e_1a_\sigma^3+2e_1e_2a_\sigma b_\sigma+e_3c_\sigma)-1+s^3<0$, so the map $H_s$ is an element of the space $\mathcal{N}_{4e}^3$.  Also $H_0=\sigma_0$ and $H_1=\sigma$, so $H$ is a path in $\mathcal{N}_{4e}^3$ from $\sigma_0$ to $\sigma$. Let us take a map $\varPsi:\zo\to\mathcal{A}_4$ given by $\varPsi(s)=\varPsi_s$ for $s\in\zo$, where $$\varPsi_s(t)=\big(\hp e_1t^2,e_2t^3+e_2(1-s)t,e_3t^4-e_3 s\hp t\hp\big)$$ for $t\in\ro$. It is easy to see that $\varPsi_0=\varphi_0$ and $\varPsi_1=\sigma_0$. For $s\in\ozo$, the map $\varPsi_s$ is an element of the space $\mathcal{N}_{4e}^1$. This shows that $\varPsi$ is a path in $\mathcal{N}_{4e}$ from $\varphi_0$ to $\sigma_0$. The maps $H$ and $\varPsi$ both together gives a path in $\mathcal{N}_{4e}$ from $\varphi_0$ to $\sigma$.
\end{proof}

\begin{theorem}\label{thm19}
The space $\pfrt$ has eight path components.
\end{theorem}

\begin{proof}
For an element $e=(e_1,e_2,e_3)$ in $\{-1,1\}^3$, we consider the following set:\vw

\hskip20mm$\mathcal{\tilde{P}}_{4e}=\big\{\hs \tau\in\mathcal{\tilde{P}}_4\mid\,e_1a_{2\tau}>0$, $e_2b_{3\tau}>0$\, and\, $e_3c_{4\tau}>0\hs\big\}$\hp.\vw

\noindent Note that $\mathcal{\tilde{P}}_4=\bigcup_e\mathcal{\tilde{P}}_{4e}$\hp. Also, for $e\neq e'$, there is no path in $\mathcal{\tilde{P}}_4$ from an element of the set $\mathcal{\tilde{P}}_{4e}$ to an element of the set $\mathcal{\tilde{P}}_{4e'}$\hp. Now it is enough to prove that for each $e\in\{-1,1\}^3$, the space $\mathcal{\tilde{P}}_{4e}$ is path connected.\vo 

Let $e=(e_1,e_2,e_3)$ be an arbitrary element in the set $\{-1,1\}^3$. From the proof of proposition \ref{thm15}, one can see that every element in $\mathcal{\tilde{P}}_{4e}$ is connected by a path in $\mathcal{\tilde{P}}_{4e}$ to some element in $\mathcal{N}_{4e}$. Since by proposition \ref{thm18}, the space $\mathcal{N}_{4e}$ is path connected, so the space $\mathcal{\tilde{P}}_{4e}$ is also path connected.
\end{proof}

%%%%%%%%%%%%%%%%%%%%%%%%%%%%%%%%%%%%%%%%%%%%%%%%%%%%%%%%%%%%%%%%%%%%%%%%%%%%%%%%%%%%%%%
\subsection{ The space $\pft$}\label{sec3.3}

\noindent\hskip6mm From proposition \ref{thm5}, we can realize only the unknot $0_1$, the left hand trefoil $3_1$ and the right hand trefoil $3_1^*$ in degree $5$. In fact, Shastri \cite{ars} had shown a realization of the trefoil knot in degree $5$.\vo 

A {\it mathematica} plot of the Shastri's trefoil $t\mapsto\big(\hp t^3-3\hp t,\hp t^4-4\hp t^2,\hp t^5-10\hp t\hp\big)$ and its mirror image $t\mapsto\big(\hp t^3-3\hp t,\hp t^4-4\hp t^2,\hp -(t^5-10\hp t)\hp\big)$ is shown in the figure bellow:
\begin{figure}[H]
\begin{center}
\includegraphics[scale=0.44]{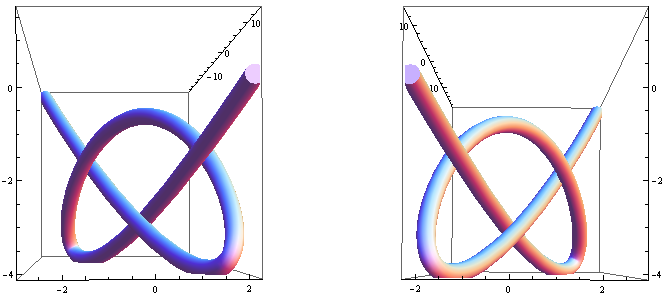}
\caption{Representations of $3_1$ and $3_1^*$}
\end{center}
\end{figure} 
The polynomial degree of the trefoil knot is $5$. The degrees $3$ and $4$ of the first and second components are minimal in the sense that there is no polynomial representation of the trefoil knot belonging to the space $\pf\setminus\pft$. By corollary \ref{cor2}, the right hand trefoil and the left hand trefoil each corresponds to at least $4$ path components of the space $\pft$. For example, the knots\\
{\small $t\mapsto\big(\hp t^3-3\hp t,\hp t^4-4\hp t^2,\hp t^5-10\hp t\hp\big)$\hp,\hw
$t\mapsto\big(-(t^3-3\hp t),\hp-(t^4-4\hp t^2),\hp t^5-10\hp t\hp\big)$\hp,\\ 
$t\mapsto\big(-(t^3-3\hp t),\hp t^4-4\hp t^2,\hp-(t^5-10\hp t)\hp\big)$\hw and\hw
$t\mapsto\big(\hp t^3-3\hp t,\hp-(t^4-4\hp t^2),\hp-(t^5-10\hp t)\hp\big)$}\\
\noindent represent the same trefoil, but they lie in the different path components of the space $\pft$. Also, by the same corollary, the unknot has $8$ path components in the space $\pft$. Thus, the space $\pft$ has at least $16$ path components corresponding to the knots $0_1, 3_1$ and $3_1^*$. We summarize the details in the following table:\vskip5mm
\begin{center}
\begin{tabular}[H]{|p{5mm}| p{12mm} | p{33mm}| p{57mm}|}
\hline Sr. No. & Knot type & Polynomial degree of a knot type & Number of path components of $\pft$ corresponding to a knot type\\ \hline
$1.$ &$0_1$ & 1 & at least 8\\
$2.$ &$3_1$& 5 & at least 4\\ 
$3.$ &$3_1^*$& 5 & at least 4\\ \hline
\multicolumn{3}{|p{60mm}|}{Number of path components of $\pft$} & at least 16\\[2.5mm]\hline
\end{tabular}
\end{center}
%%%%%%%%%%%%%%%%%%%%%%%%%%%%%%%%%%%%%%%%%%%%%%%%%%%%%%%%%%%%%%%%%%%%%%%%%%%%%%%%%%%%%%%
\subsection{ The space $\pst$}\label{sec3.4}

\noindent\hskip6mm The knots which have polynomial representation in degree $5$ naturally have their representation in degree $6$ as well. By proposition \ref{thm2}, for a knot $\kappa$ having a polynomial representation in degree $6$, the minimal crossing number $c[\kappa]$ must be less than or equal to $6$. Since the knots $5_1, 5_1^*, 3_1\#3_1, 3_1^*\#3_1^*$ and $3_1\#3_1^*$ are $4$-superbridge, so they can not be represented in degree $6$ (see proposition \ref{thm2}). Also, by theorem \ref{thm7}, it is almost impossible to represent the knots $5_2$ and $5_2^*$ in the space $\pst$. The same is true for the knots $6_1, 6_1^*, 6_2, 6_2^*$ and $6_3$. But, we can represent the figure-eight knot in the space $\pst$. In fact, we have a polynomial representation $t\mapsto\fght$ of the figure-eight knot ($4_1$ knot) with degree sequence $(4, 5, 6)$, where{\footnotesize
\begin{eqnarray*}
f(t)&=& (-4.8 + t)\ho  (-0.3 + t)\ho  (3.6 + t)\ho  (10 + t)\hp,\\
g(t)&=& (-4.8 + t)\ho  (-3.3 + t)\ho  (-0.3 + t)\ho  (2.3 + t)\ho  (4.6 + t)\hw\mbox{and}\\
h(t)&=& 0.5\ho  t\ho  (-0.19 + t)\ho  (21.22 - 9.19\ho  t + t^2)\ho  (17.78 + 8.42\ho  t + t^2)\hp.
\end{eqnarray*}}
A {\em mathematica} plot of this representation is shown in the following figure:
\begin{figure}[H]\label{fig3}
	\begin{center}
		\includegraphics[scale=0.32]{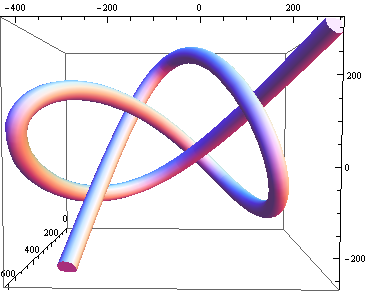}
		\caption{Representation of $4_1$}
	\end{center}
\end{figure}
By proposition \ref{thm5}, it follows that the polynomial degree of the figure-eight knot is $6$. Note that in the polynomial representation of this knot, the degrees $4$ and $5$ of the first and second components are minimal in the sense that there is no  polynomial representation of the figure-eight knot belonging to the space $\ps\setminus\pst$. By corollary \ref{cor2}, the space $\pst$ has at least $8$ path components corresponding to the figure-eight knot. Also, the knots $0_1, 3_1$ and $3_1^*$ can be realized in $\pst$ (since they have representation in $\pft$). The unknot $0_1$ corresponds to at least $8$ path components of the space $\pst$. The right hand trefoil $3_1$ and the left hand trefoil $3_1^*$, each corresponds to at least $4$ path components of the space $\pst$. Hence the space $\pst$ has at least $24$ path components. We summarize the details in the table bellow:\vskip5mm
\begin{center}
	\begin{tabular}[H]{|p{5mm}| p{12mm} | p{33mm}| p{57mm}|}
		\hline Sr. No. & Knot type & Polynomial degree of a knot type & Number of path components of $\pst$ corresponding to a knot type\\ \hline
		$1.$ &$0_1$ & 1 & at least 8\\ 
		$2.$ &$3_1$& 5 & at least 4\\ 
		$3.$ &$3_1^*$& 5 & at least 4\\ 
		$4.$ &$4_1$& 6 & at least 8\\ \hline
		\multicolumn{3}{|p{60mm}|}{Number of path components of $\pst$} & at least 24\\[2.5mm] \hline
	\end{tabular}
\end{center}

%%%%%%%%%%%%%%%%%%%%%%%%%%%%%%%%%%%%%%%%%%%%%%%%%%%%%%%%%%%%%%%%%%%%%%%%%%%%%%%%%%%%%%%
\subsection { The space $\psvt$}\label{sec3.5}

\noindent\hskip6mm By proposition \ref{thm2}, for a knot having a polynomial representation in degree $7$, the minimal crossing number must be less than or equal to $10$. In fact, we have produced some polynomial representations of the knots $5_1, 5_1^*, 5_2, 5_2^*, 6_1, 6_1^*, 6_2, 6_2^*, 6_3,$ $3_1\#3_1, 3_1^*\#3_1^*, 3_1\#3_1^*, 8_{19}$ and $8_{19}^*$ in the space $\psvt$.\vo 
\noindent 1) A polynomial representation $t\mapsto\uvwt$ of the knot $5_1$ with degree sequence $(5, 6, 7)$ is given by{\footnotesize 
\begin{eqnarray*}
u(t)&=& 0.00001\,t^5 + 4\ho  (-24.01 + t^2)\ho  (-4 + t^2)\hp,\\
v(t)&=& 0.00001\,t^6 + t\ho  (-30.25 + t^2)\ho  (-12.25 + t^2)\hw\mbox{and}\\
w(t)&=& -\ho  0.1\ho  t\ho  (-26.8328 + t^2)\ho  (-13.6702 + t^2)\ho (0.1135 + t^2)
\end{eqnarray*}}
\begin{figure}[H]
	\begin{center}
		\includegraphics[scale=0.32]{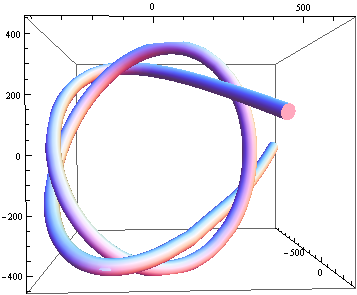}
		\caption{Representation of $5_1$}
	\end{center}
\end{figure}
\noindent 2) A polynomial representation $t\mapsto\xyzt$ of the knot $5_2$ with degree sequence $(5, 6, 7)$ is given by{\footnotesize 
\begin{eqnarray*} 
x(t)&=& 0.00001\,t^5 + 20\ho(-17 + t)\ho(-10 + t)\ho(15 + t)\ho(21 + t)\hp,\\
y(t)&=& 0.00001\,t^6 + t\ho  (-400 + t^2)\ho  (-121 + t^2)\hw\mbox{and}\\
z(t)&=& -\ho  0.005\ho  t\ho  (-20.1133216 + t)\ho (0.0107598 - 0.0343124 \ho   t + t^2)\\ 
    & & (12.2430449 + t)\ho  (20.5785825 + t)\ho (-14.260128 + t)\hp.
\end{eqnarray*}}
\begin{figure}[H]
	\begin{center}
		\includegraphics[scale=0.32]{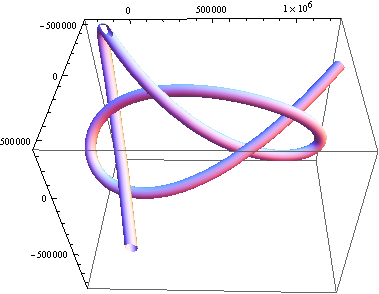}
		\caption{Representation of $5_2$}
	\end{center}
\end{figure}
\noindent 3) A polynomial representation $t\mapsto\fght$ of the knot $6_1$ with degree sequence $(5, 6, 7)$ is given by{\footnotesize  
\begin{eqnarray*} 
f(t)&=& 60\ho  (-43.4 + t)\ho  (-28 + t)\ho  (5 + t)\ho  (31.4 + t)\ho  (47.6 + t)\hp,\\
g(t)&=& (-49 + t)\ho  (-38 + t)\ho  (-8 + t)\ho  (-6 + t)\ho  (28 + t)\ho  (43.6 + t)\hw\mbox{and}\\
h(t)&=& -\ho  0.07\ho  (-45.995024874 + t)\ho  (5.231021635 + t) \ho (758.763745443 - 54.4650519227\ho  t + t^2)\\ & & (19.036560084 + t)\ho (2059.948386689 + 90.4819595699\ho  t + t^2)\hp.
\end{eqnarray*}}
\begin{figure}[H]
	\begin{center}
		\includegraphics[scale=0.32]{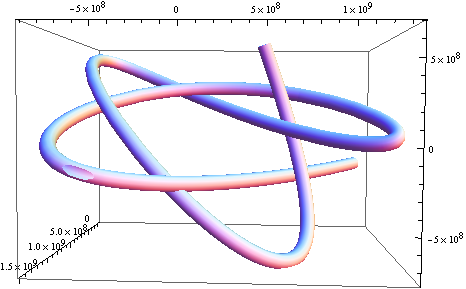}
		\caption{Representation of $6_1$}
	\end{center}
\end{figure}
\noindent 4) A polynomial representation $t\mapsto\uvwt$ of the knot $6_2$ with degree sequence $(5, 6, 7)$ is given by{\footnotesize 
\begin{eqnarray*}
u(t)&=& 4\ho  (-39 + t)\ho  (-5 + t)\ho  (35 + t)\ho  (-625 + t^2)\hp,\\
v(t)&=& 0.1\ho  (-39 + t)\ho  (-30 + t)\ho  (-10 + t)\ho  (20 + t)\ho  (25 + t)\ho  (41 + t)\hw\mbox{and}\\
w(t)&=& 0.005\ho  t\ho  (-39.8753791 + t)\ho  (-27.4156408 + t)\ho  (28.436878 + t)\\ && (37.25572585 + t)\ho (0.002423881 - 0.005429486\ho  t + t^2)\hp.
\end{eqnarray*}}
\begin{figure}[H]
	\begin{center}
		\includegraphics[scale=0.32]{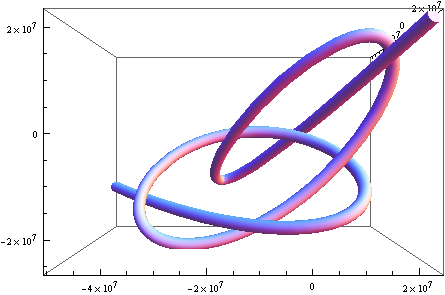}
		\caption{Representation of $6_2$}
	\end{center}
\end{figure}
\noindent 5) A polynomial representation $t\mapsto\xyzt$ of the knot $6_3$ with degree sequence $(5, 6, 7)$ is given by{\footnotesize 
\begin{eqnarray*}
x(t)&=& 15\ho  (-29 + t)\ho  (-20 + t)\ho  (10 + t)\ho  (30 + t)^2\hp,\\
y(t)&=& (-32 + t)\ho  (-6 + t)\ho  (4 + t)\ho  (30 + t)\ho  (-400 + t^2)\hw\mbox{and}\\
z(t)&=& -\ho  0.06\ho  (376.737563885 - 37.8892469397\ho  t + t^2)\ho (144.275534095 + 21.404400212\ho  t + t^2)\\ & & (-33.329044815 + t)\ho (955.985733648 + 61.56649851\ho  t + t^2)\hp. 
\end{eqnarray*}}
\begin{figure}[H]
	\begin{center}
		\includegraphics[scale=0.32]{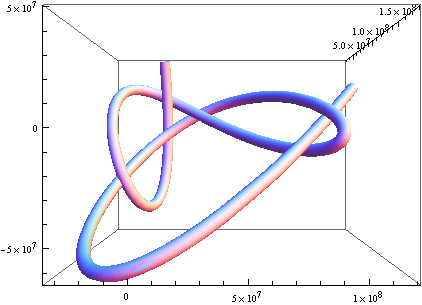}
		\caption{Representation of $6_3$}
	\end{center}
\end{figure}
\noindent 6) A polynomial representation $t\mapsto\fght$ of $3_1\#3_1$ knot with degree sequence $(5, 6, 7)$ is given by{\footnotesize  
\begin{eqnarray*}
f(t)&=& 5\ho t\ho (77.3 - 17.5\ho  t + t^2)\ho (77.3 + 17.5\ho  t + t^2)\hp,\\
g(t)&=& (-102.01 + t^2)\ho  (-53.29 + t^2)\ho  (-4.84 + t^2)\hw\mbox{and}\\
h(t)&=& -\ho  0.15\ho  t\ho  (-99.695462027 + t^2)\ho  (-68.11720396 + t^2)\ho  (0.025367747 + t^2)\hp.
\end{eqnarray*}}
\begin{figure}[H]
	\begin{center}
		\includegraphics[scale=0.32]{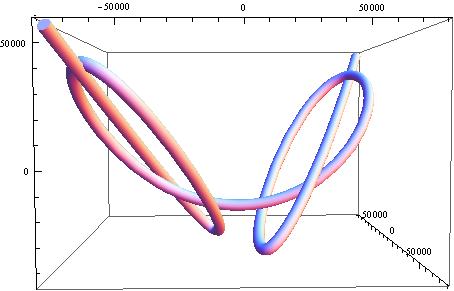}
		\caption{Representation of $3_1\#3_1$}
	\end{center}
\end{figure}
\noindent 7) A polynomial representation $t\mapsto\uvwt$ of $3_1\#3_1^*$ knot with degree sequence $(5, 6, 7)$ is given by{\footnotesize
\begin{eqnarray*}
u(t)&=& 30\ho  (-32.5 + t)\ho  (-21.3 + t)\ho  (-3.3 + t)\ho  (16.2 + t)\ho  (28 + t)\hp,\\
v(t)&=& (-34 + t)\ho  (-23 + t)\ho  (-6.8 + t)\ho  (12 + t)\ho  (21.7 + t)\ho  (33.1 + t)\hw\mbox{and}\\
w(t)&=& -\ho 0.03\ho  t\ho  (-32.807367 + t)\ho  (-24.209735 + t)\ho (15.257278 + t)\\ & & (28.289226 + t)\ho (0.0043718 - 0.0082068\ho  t + t^2)\hp.
\end{eqnarray*}}
\begin{figure}[H]
	\begin{center}
		\includegraphics[scale=0.32]{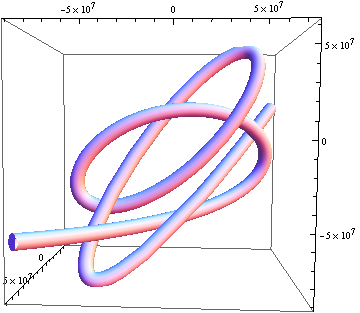}
		\caption{Representation of $3_1\#3_1^*$}
	\end{center}
\end{figure}
\noindent 8) A polynomial representation $t\mapsto\xyzt$ of the knot $8_{19}$ with degree sequence $(5, 6, 7)$ is given by{\footnotesize
\begin{eqnarray*}
x(t)&=& t^5 - 5.5\ho t^3 + 4.5\ho t\hp,\\
y(t)&=& t^6 - 7.35\ho t^4 + 14\ho t^2\hw\mbox{and}\\
z(t)&=& t^7 - 8.13297\ho t^5 + 18.5762\ho  t^3 - 10.4337\ho  t\hp.
\end{eqnarray*}}
\begin{figure}[H]
	\begin{center}
		\includegraphics[scale=0.35]{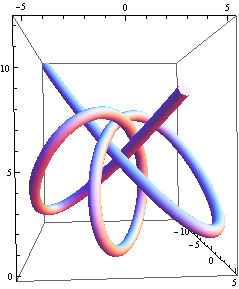}
		\caption{Representation of $8_{19}$}
	\end{center}
\end{figure}
Each of the knot $5_1, 5_1^*, 3_1\#3_1, 3_1^*\#3_1^*, 3_1\#3_1^*,$ $8_{19}$ and $8_{19}^*$ is $4$-superbridge, so by proposition \ref{thm2}, one can not represent any one by a polynomial knot with degree less than $7$. In other words, each has polynomial degree $7$. However, the polynomial degree the knots $5_2, 5_2^*, 6_1, 6_1^*, 6_2, 6_2^*$ and $6_3$ is either $6$ or $7$. For the polynomial representation of $8_{19}$, the degrees $5$ and $6$ of the first and second components are minimal in the sense that there is no representation of the knot $8_{19}$ belonging to the space $\psv\setminus\psvt$. The space $\psvt$ has at least $88$ path components corresponding to the knots $0_1, 3_1, 3_1^*,4_1, 5_1, 5_1^*, 5_2, 5_2^*,6_1, 6_1^*, 6_2, 6_2^*, 6_3, 3_1\#3_1, 3_1^*\#3_1^*, 3_1\#3_1^*, 8_{19}$ and $8_{19}^*$. A comprehensive table of an estimation of the number of path components of the space $\psvt$ is given below:\vskip5mm
\begin{center}
	\begin{tabular}[H]{|p{5mm}| p{12mm} | p{33mm}| p{57mm}|}
		\hline Sr. No. & Knot type & Polynomial degree of a knot type & Number of path components of $\psvt$ corresponding to a knot type\\ \hline
		$1.$ &$0_1$ & 1 & at least 8\\
		$2.$ &$3_1$& 5 & at least 4\\
		$3.$ &$3_1^*$& 5 & at least 4\\
		$4.$ &$4_1$& 6 & at least 8\\
		$5.$ &$5_1$& 7 & at least 4\\
		$6.$ &$5_1^*$& 7 & at least 4\\
		$7.$ &$5_2$& 6 or 7 & at least 4\\ 
		$8.$ &$5_2^*$& 6 or 7 & at least 4\\
		$9.$ &$6_1$& 6 or 7 & at least 4\\
		$10.$ &$6_1^*$& 6 or 7 & at least 4\\
		$11.$ &$6_2$& 6 or 7 & at least 4\\ 
		$12.$ &$6_2^*$& 6 or 7 & at least 4\\ 
		$13.$ &$6_3$& 6 or 7 & at least 8\\ 
		$14.$ &$3_1\#3_1$& 7 & at least 4\\ 
		$15.$ &$3_1^*\#3_1^*$& 7 & at least 4\\ 
		$16.$ &$3_1\#3_1^*$& 7 & at least 8\\ 
		$17.$ &$8_{19}$& 7 & at least 4\\ 
		$18.$ &$8_{19}^*$& 7 & at least 4\\ \hline
		\multicolumn{3}{|p{60mm}|}{Number of path components of $\psvt$} & at least 88\\[2.5mm]\hline
	\end{tabular}
\end{center}

%%%%%%%%%%%%%%%%%%%%%%%%%%%%%%%%%%%%%%%%%%%%%%%%%%%%%%%%%%%%%%%%%%%%%%%%%%%%%%%%%%%%%%%
\section{Conclusion}\label{sec4}

\noindent\hskip6mm  We have seen that the space $\pfr$ is path connected where as the space $\pfrt$ has eight path components. Thus, it makes a difference if we are considering a space with fixed degrees of components polynomials or with a flexible range of degrees.  We also see that the space $\spk$ of all polynomial knots can be stratified in two different ways such as $$\spk=\bigcup _{d\geq 1} K_d=\bigcup _{d\geq 2} O_d,$$  where $\od$ is the space of all polynomial knots $t\mapsto (f(t),g(t),h(t))$ with $\deg(f)\leq d-2$, $\deg(g)\leq d-1$ and $\deg(h)\leq d.$ We can provide the inductive limit topology using any of these stratifications. It needs to be observed that the number of path components in these resulting spaces are different. For a fixed $d$, there are many more interesting spaces of polynomial knots with conditions on the degrees of the component polynomials. Each one of them gives interesting topology. One may try to find a suitable topology on $\spk$ such that its path components correspond precisely to  knot types.  In general  we might like to study the spaces $\mathcal{K}_{p,q,r}$  as the set of all polynomial knots $t\to (f(t),g(t),h(t))$ with $\deg(f)=p,\deg(g)=q, \deg(h)=r$ where $p,q$ and $r$ are any given positive integers.  We would like to explore if the number of path components in $\mathcal{K}_{p,q,r}$ corresponding to a knot type gets affected in case $(p,q,r)$ is the minimal degree sequence for a given knot type. 

%%%%%%%%%%%%%%%%%%%%%%%%%%%%%%%%%%%%%%%%%%%%%%%%%%%%%%%%%%%%%%%%%%%%%%%%%%%%%%%%%%%%%%%

%%%%%%%%%%%%%%%%%%%%%%%%%%%%%%%%%%%%%%%%%%%%%%%%%%%%%%%%%%%%%%%%%%%%%%%%%%%%%%%%%%%%%%%%%
\end{document}